\documentclass[final]{elsarticle}
\usepackage[latin1]{inputenc}
\usepackage{amsmath,bm}
\usepackage{amsthm}
\usepackage{amsfonts}
\usepackage{color}
\usepackage[dvipsnames]{xcolor}
\usepackage{amssymb}
\usepackage{graphicx}
\usepackage{comment}
\usepackage{caption}
\usepackage{subcaption}
\usepackage{empheq}
\usepackage[top=1in, bottom=1in, left=1in, right=1in]{geometry}
\usepackage{mathrsfs}

\usepackage[pdftex, unicode=true,
colorlinks,
linkcolor=red,
anchorcolor=blue,
citecolor=green
]{hyperref}

\newtheorem{thm}{Theorem}[section]
\newtheorem{lem}[thm]{Lemma}

\newtheorem{remark}[thm]{Remark}

\newcommand{\bff}{\bm f}

\newcommand{\bw}{\bm w}

\newcommand{\bu}{\bm u}

\newcommand{\bv}{\bm v}

\newcommand{\bV}{\bm V}

\newcommand{\bR}{\bm R}

\newcommand{\Rey}{\ensuremath{\mathrm{Re}}}

\numberwithin{equation}{section}

\makeatletter
\def\ps@pprintTitle{%
	\let\@oddhead\@empty
	\let\@evenhead\@empty
	\def\@oddfoot{}%
	\let\@evenfoot\@oddfoot}
\makeatother

\begin{document}
	\begin{frontmatter}
		\title{Convergence analysis of dynamically regularized Lagrange multiplier pressure correction method for the incompressible Navier-Stokes equations}
		
		\author[ouc]{Yi Shen}
		\ead{shenyi@stu.ouc.edu.cn}
		\author[ouc,lmm]{Rihui Lan\corref{cor}}
		\ead{lanrihui@ouc.edu.cn}
		\author[xtu]{Hua Wang}
		\ead{ wanghua@xtu.edu.cn}

		\address[ouc]{School of Mathematical Sciences, Ocean University of China, Qingdao, Shandong 266100, China}
		\address[lmm]{Laboratory of Marine Mathematics, Ocean University of China, Qingdao, Shandong 266100, China}
		\address[xtu]{School of Mathematics and Computational Science, Xiangtan University, Xiangtan, 411105,  China}

		\cortext[cor]{Corresponding author}
		
		\date{July 2025}
		\begin{abstract}
			\textcolor{blue}{We propose first-order pressure-correction scheme for the incompressible Navier-Stokes equations, incorporating the recently developed the Dynamically Regularized Lagrange Multiplier (DRLM) methods.} The resulting algorithms are  fully decoupled and require solving only  Poisson-type equations at each time step. Moreover, it exhibits unconditional energy stability.
			This paper provides a rigorous error analysis for the first-order scheme, establishing optimal error estimates for both velocity and pressure. Specifically, we employ mathematical induction to derive sharp velocity error bounds, while leveraging the inf-sup condition to prove optimal convergence rate for the pressure. To validate our theoretical findings, we present two numerical experiments demonstrating the accuracy and robustness of the method.
		\end{abstract}
		
	\end{frontmatter}
	
	\section{Introduction}
	\textcolor{blue}{The incompressible Navier-Stokes (NS) equations, the fundamental mathematical model for fluid flow simulation, are posed on a open bounded domain \(\Omega \subset \mathbb{R}^d (d = 2, 3)\) over a time interval $(0,T]$ with $T > 0$, and take the form:}
	\begin{subequations}\label{NS:orig}
		\begin{empheq}[left=\empheqlbrace]{align}
			&\textstyle\bu_t - \nu \Delta\bu + (\bu \cdot \nabla) \bu + \nabla p =\bff, & \text{in } \Omega \times (0, T],\label{NS:orig:mom} \\
			&\textstyle\nabla \cdot \bu = 0, & \text{in } \Omega \times (0, T],\label{NS:orig:incom}
		\end{empheq}
	\end{subequations}
	where $\bu=\bu(x,t)$ is the velocity and $p=p(x,t)$ is the pressure field. Here, \(\nu=1/\Rey\) represents the kinematic viscosity and $\Rey>0$ denotes the Reynolds number. Throughout, we assume $\nu<1$. The system is subject to the initial condition \(\bu(\cdot, 0) = \bu_0\) and the no-slip boundary condition, i.e. \(\bu|_{\partial \Omega} = 0\). 
	
	The NS equations are fundamental in many engineering applications, including aerodynamics and bio-fluid mechanics, as well as in scientific computing. As such, considerable research has been dedicated to developing and analyzing numerical methods for their approximation \cite{Chorin1968, GLOWINSKI2003, Shen1992SIAM}. However, solving these equations numerically remains highly challenging due to several key difficulties. 
	A major obstacle stems from the pressure, which lacks an independent evolution equation and must instead be inferred indirectly through the incompressibility constraint $\nabla \cdot \bu = 0$. This results in a strong coupling between velocity and pressure, complicating numerical schemes. Additionally, the nonlinear convection term introduces further complexity, making the numerical treatment even more demanding.

	%The NS equations are of broad significance in both engineering applications, such as aerodynamics and biofluid mechanics, and scientific computing. Consequently, extensive research has been devoted to the design and analysis of numerical methods for their approximation, as evidenced by \cite{Chorin1968, GLOWINSKI2003, Shen1992SIAM} and the references therein. Nevertheless, the numerical solution of these equations remains particularly challenging. One major difficulty arises from the fact that the pressure, which lacks an independent evolution equation, must be determined indirectly by the incompressibility constraint $\nabla \cdot \bu = 0$, resulting in a strong coupling between velocity and pressure. Additionally, the nonlinear convection term introduces further complexity and complicates the numerical treatment.

	To address the velocity--pressure coupling issue, numerous numerical methods have been proposed. Among these, the projection method, pioneered by Chorin \cite{Chorin1968} and Temam \cite{Temann1969}, has become a predominant approach for solving the NS equations. 
	This method leverages the Helmholtz decomposition, splitting the computation into two steps: (i) 
	a prediction step that advances the velocity while temporarily relaxing incompressibility; (ii) 
	a projection step that enforces incompressibility ($\nabla \cdot \bu = 0$). 
	Due to its sequential handling of velocity and pressure, the projection method is highly versatile, allowing compatibility with various discretization schemes. 
	Critically, it avoids solving a fully coupled system at each time step, thereby significantly reducing computational cost. 
	Specifically, the projection step involves solving only a Poisson equation for pressure, while the prediction step can be treated independently, which is a key efficiency advantage.
	
	%In response to the coupling between the velocity and pressure, a variety of numerical methods have been developed.  Among them, the projection method, first introduced by Chorin~\cite{Chorin1968} and Temam~\cite{Temann1969}, has emerged as one of the mainstream approaches for solving the  NS equations. Such method utilizes the Helmholtz decomposition. It decomposes the computation into a prediction step for the velocity and a projection step that enforces incompressibility, making it compatible with various discretization frameworks. By decoupling the velocity and pressure through a sequential solution strategy, the projection method significantly reduces computational cost and avoids the need to solve a fully coupled system. In particular, the projection step requires only the solution of a Poisson equation for the pressure, while the prediction step can be treated independently.

	The nonlinear convection term plays a pivotal role in governing the energy stability of the system. While implicit schemes \cite{CushmanRoisinBeckers2011,GRAUSANCHEZ2011} provide theoretical unconditional stability, they incur significant computational overhead, as they require solving a nonlinear system at each time step. This process involves nonlinear iterative solvers (e.g., Newton's method), the costly assembly and linearization of Jacobian matrices, complicating implementation. Furthermore, despite their theoretical stability guarantees, an improper numerical treatment of nonlinear terms can still result in energy conservation violations in practice.
	Conversely, explicit schemes \cite{Thomas1995} are computationally simpler, but suffer from numerical oscillations or even divergence when the time step exceeds stability limits, owing to their explicit handling of convection and viscous terms. As an intermediate approach, semi-implicit schemes, which treat convection velocity explicitly and gradient of velocity implicitly, remain challenging because they require solving a variable-coefficient Stokes-like  equation coupled to the velocity field. 
	Meanwhile, stabilization is often required. For example, \cite{labovsky2009stabilized} proposed adding the term $-\alpha h\Delta \bu$ into their semi-implicit scheme, where $\alpha$ serves as a tuning parameter and $h$ is the mesh size.
	Hence, to ensure unconditional energy stability, carefully designed strategies are still required.
	%advanced numerical strategies still need to be carefully employed. %Otherwise, the explicit treatment of convection velocity can still undermine stability.
	
	%As for the nonlinear convection term, it is crucial for the energy stability. Although implicit schemes \cite{CushmanRoisinBeckers2011,GRAUSANCHEZ2011} offer theoretical unconditional stability, their computational cost increases substantially since a nonlinear system must be solved at each time step. This typically involves complex nonlinear iterations (e.g., Newton's method), along with the assembly and linearization of Jacobian matrices, which complicates implementation. Moreover, even with theoretical stability guarantees, inadequate treatment of nonlinear terms in practice may still lead to violations of energy conservation. On the other hand, explicit schemes\cite{Thomas1995} are simpler to implement but are prone to numerical oscillations or divergence when the time step is too large, due to explicit treatment of convection and viscous terms. As a compromise, semi-implicit schemes, which treat convection explicitly and diffusion implicitly, still require solving a variable-coefficient coupled elliptic equation that depends on the velocity field. This necessitates careful numerical design, such as the use of Scalar Auxiliary Variable (SAV) methods, to ensure unconditional energy stability. Otherwise, explicit convection terms may still adversely affect stability.
	
	%%%%%%%%%%%%%%%%%%%%%%%%%%%%%%%%%%%%%%%%%%%%%%%%%%%
	Recent advances in numerical methods for dissipative systems focus on preserving the energy dissipation law by introducing auxiliary equations to control the energy evolution of the system. A notable example is the Scalar Auxiliary Variable (SAV) method \cite{SHEN2018JCp, Shen2019SIAM}, originally developed for gradient flow models. This approach provides unconditional energy stability while allowing explicit treatment of nonlinear terms, offering significant computational advantages. Building on this, \cite{lin2019numerical} extended the SAV method to the NS equations, 
	and by introducing an auxiliary scalar variable, transformed the nonlinear system into a linear one, thus enabling explicit handling of convection terms.
	%where an auxiliary scalar variable transforms the nonlinear system into a linear one, enabling explicit handling of convection terms.
	Despite its efficiency, the SAV method has its limitations: (i) the modified energy it preserves lacks a direct physical connection to the true energy of the original system, and (ii) its accuracy becomes sensitive to parameter choices (e.g., truncation thresholds for the auxiliary variable) for larger time steps. Improper parameter selection may compromise the physical fidelity of simulations.
	
	%Recently, some methods incorporate the energy dissipation law directly into the whole system by introducing an extra equation to control the evolution of the system "energy" when the nonlinear term is treated explicitly. The Scalar Auxiliary Variable (SAV) method \cite{SHEN2018JCp, Shen2019SIAM} was developed to handle nonlinear terms in gradient flow models and has shown great benefit and convenience for unconditional energy stability. Then, \cite{lin2019numerical} extend the SAV method to solve the NS equations. By introducing a dynamic equation on the extra scalar variable, it transforms the nonlinear system into a linear one, thereby allowing explicit treatment of convection terms. Although SAV improves computational efficiency, the resulting energy is a modified quantity that lacks a direct physical interpretation relative to the true energy of the original system. Moreover, for larger time steps, the accuracy of the method becomes sensitive to parameters such as the truncation threshold of the auxiliary variable. Improper parameter choices may lead to inaccurate physical representations.
	
	Taking inspiration from the SAV approach, the Lagrange multiplier (LM) method \cite{Cheng2020CMAME,yang2022original} was developed to construct an unconditionally energy stable scheme that preserves the original energy law. In this method, the energy constraint is enforced via a Lagrange multiplier, ensuring consistency with the physical energy dynamics. However, the scheme requires restrictive time step sizes to maintain stability due to the non-uniqueness of the multiplier, often necessitating additional post-processing truncation techniques \cite{cheng2025unique}.
	To address this limitation, recent work introduced the dynamically regularized LM (DRLM) method \cite{Doan2025JCP,doan2025dynamically}, which imposes uniqueness on the multiplier by adding the energy equation with a time derivative on the square of multiplier. This regularization enables stable simulations even at large time steps. Nevertheless, current DRLM methods  still require solving a coupled velocity-pressure  system, inheriting the computational challenges of incompressible NS simulations.
	
	%Taking inspiration from the SAV method, the Lagrange multiplier method \cite{Cheng2020CMAME,yang2022original} was proposed to produce an unconditionally energy-stable scheme related to the original energy. It treats the original energy law as a constraint of the system, and incorporates the law into the system via a Lagrange multiplier. However, the proposed scheme requires a small time step to maintain its stability, stemming from the absence of uniqueness in the Lagrange multiplier.  Usually, it requires the extra posterior truncation \cite{cheng2025unique}. Recently, \cite{Doan2025JCP,doan2025dynamically} proposed a dynamically regularized LM (DRLM) method to impose the uniqueness of the multiplier, so that the scheme can provide a stable simulation with a large time step.  DRLM methods introduce an extra time derivative on the square of multiplier to regularize the system. However, it should be noted that the DRLM methods in \cite{Doan2025JCP} still solves the velocity-pressure coupling system, which remains an inherent challenge in the numerical solution of the incompressible NS equations.
	
	%-----------------------
	{\color{blue}To further enhance the computational efficiency}, we develop a novel DRLM-based projection method in this work that combines the advantages of dynamic regularization while yielding an unconditionally energy-stable and {\color{blue} velocity-pressure} decoupled system. Our primary contributions are: 
	\begin{itemize}
		\item [(i)]
		{\color{blue}a new DRLM-based scheme that incorporates the first-order standard-incremental pressure-correction approach. By introducing the proper energy, the proposed schemes exhibit unconditional energy stability and guarantee the unconditional solvablity on the Lagrange multiplier. Moreover, it can be efficiently solved by several  Poisson equations and a quadratic algebraic equation;}
		%eliminates the need to solve nonlinear algebraic systems,  and the second-order rotational 
		\item [(ii)]
		a comprehensive numerical analysis framework featuring rigorous error estimates for the first-order scheme. We establish optimal convergence rates for the velocity, pressure, and Lagrange multiplier approximations under appropriate regularity assumptions.
	\end{itemize}
	The resulting methodology preserves the computational efficiency of projection-type methods while inheriting the theoretical guarantees of DRLM techniques.
	
	%In this paper, we are going to design the  DRLM based projection method, so that the proposed scheme can inherit their advantages, and obtain a linear, unconditional energy stable and decoupled system. In addition, we will provide a rigorous error analysis on the proposed scheme. The key contributions are as follows: (i) we develop a new DRLM-based scheme incorporating first-order standard incremental pressure correction and second-order rotational pressure correction. The scheme is fully linear, avoids solving nonlinear algebraic equations, and exhibits improved energy stability; (ii) we establish a rigorous error analysis for the first-order scheme, deriving optimal error estimates for the velocity, pressure, and Lagrange multiplier.

	The remainder of this paper is organized as follows. Section \ref{scheme} introduces the DRLM-based pressure-correction method, establishing rigorous energy stability results for the first-order scheme and analyzing some key properties of the numerical solutions. Section \ref{errest} presents optimal temporal error estimates for the velocity, Lagrange multiplier, and pressure. Numerical experiments are provided in Section \ref{num} to validate the theoretical results. We conclude in Section \ref{sect:con} with a summary of key findings and perspectives for future work.
	
	%The remainder of this paper is organized as follows. In Section 2, we introduce the DRLM-based pressure-correction methods, prove  unconditional energy stability for first- and second-order schemes and analyze some key properties of the numerical solutions. Section 3 presents optimal temporal error estimates for the velocity, Lagrange multiplier, and pressure. Numerical experiments are provided in Section 4 to validate the theoretical results. Finally, we address some conclusions in Section 5.
	
	\section{Pressure-correction DRLM scheme and its properties}\label{scheme}
	In this section, we first review the DRLM schemes proposed in \cite{Doan2025JCP}, which are equivalent to the original NS equations \eqref{NS:orig} at the continuous level. Starting from DRLM schemes, we introduce the projection method to establish a first-order pressure-correction DRLM scheme and analyze its fundamental properties. In particular, we prove the unconditional solvability of the Lagrange multiplier and regularity results for the DRLM scheme \eqref{discretization:1or}, which provide a crucial foundation for the subsequent error analysis. 
	In this paper, 
	we mainly focus on the first-order scheme and briefly present the second-order scheme and its unconditional energy stability in Remark \ref{2nd:stab}.
	%While the DRLM framework can be extended to higher-order schemes, conducting a rigorous theoretical analysis for these extensions presents significant challenges, which remain unresolved.
	\subsection{Pressure-correction  DRLM scheme}  
	Let $L^2(\Omega)$, $H^k(\Omega)$ and $H_0^k(\Omega)$ denote the standard Sobolev spaces over $\Omega$. The norm on $H^k(\Omega)$ is indicated by $\Vert\cdot\Vert_k$. For $L^2(\Omega)$, we denote by $(\cdot,\cdot)$ and $\Vert\cdot\Vert$ the inner product and the associated norm, respectively. Since the pressure is unique up to an additive constant in the NS equations, we define $H^k(\Omega)/\mathbb{R}$ as the quotient space consisting of equivalence classes of elements of $H^k(\Omega)$ differing by constants. Denote by $H$ and $V$ the following Hilbert spaces:
	\begin{equation*}
		\boldsymbol{H}=\{\bu\in\boldsymbol{L}^{2}(\Omega):\nabla\cdot\bu=0,\bu\cdot\boldsymbol{n}|_{\partial\Omega}=0\},\quad\boldsymbol{V}=\{\boldsymbol{v}\in\boldsymbol{H}_{0}^{1}(\Omega):\nabla\cdot\boldsymbol{v}=0\}.
	\end{equation*}
	Define the kinematic energy $E(\bu)=\frac12\int_{\Omega}|\bm u|^2dx$. The NS equations \eqref{NS:orig} satisfy the following energy law:
	\begin{equation}\label{stability}
		\frac{dE(\bu)}{dt} + \nu \Vert \nabla \bu\Vert^2=(\boldsymbol{f},\bu).
	\end{equation}
	DRLM approach  incorporates the energy law into the system, introduce the Lagrange multiplier $\mathcal{Q}(t)\equiv 1$ 
	and a regularization parameter $\theta > 0$, thereby reformulating the NS equations \eqref{NS:orig} as follows:
	\begin{subequations}
		\begin{empheq}[left=\empheqlbrace]{align}
			\textstyle\bu_{t}-\boldsymbol{\nu}\Delta\bu+\mathcal{Q}(\bu\cdot\nabla)\bu+\nabla p=\boldsymbol{f}, & \quad\mathrm{in~}\Omega\times(0,T], \\
			\textstyle\nabla\cdot\bu=0, & \quad\mathrm{in~}\Omega\times(0,T], \\
			\textstyle\frac{dE(\bu)}{dt} +\theta\frac{d\mathcal{Q}^{2}}{dt}=(\bu_{t}+\mathcal{Q}(\bu\cdot\nabla)\bu,\bu), & \quad\mathrm{in~}(0,T], & & & \label{energy:ns}
		\end{empheq}
	\end{subequations}
	where \eqref{energy:ns} is derived from the energy law \eqref{stability} with a regularized term $\theta\frac{d\mathcal{Q}^{2}}{dt}$. 
	
	To design the DRLM scheme in the pressure correction fashion, we will first study the energy law of the pressure correction methods. First of all, we consider a uniform partition of the time interval $[0,T]: 0 = t_0 < t_1 < \cdots < t_N = T$ with the time step size $\tau = T/N$.
	Then, the first-order incremental pressure-correction scheme \cite{goda1979multistep,Li2022MOC}  is given as follow: given $(\bm u^n, p^n)$, find $(\bm u^{n+1}, p^{n+1})$ such that
	\begin{subequations}\label{first:prj}
		\begin{empheq}[left=\empheqlbrace]{align}
			&\textstyle\frac{\hat{\bm u}^{n+1}-\bm u^n}{\tau}+\nabla  p^n-\nu\Delta\hat{\bm u}^{n+1}=0, \label{prj:mom:grad}\\ 
			&\textstyle\frac{\bm u^{n+1}-\hat{\bm u}^{n+1}}{\tau}+\nabla( p^{n+1}-  p^n)=0,\label{prj:proj:grad}\\
			&\textstyle\nabla \cdot \bm u^{n+1}=0,\label{prj:incom:grad}
		\end{empheq}
	\end{subequations}
	where the nonlinear term and the external force are ignored for simplicity. \textcolor{blue}{ Eq. \eqref{prj:mom:grad} is the prediction step by ignoring the incompressible condition, and $\hat{\bm u}^{n+1}$ is the intermediate value which is not divergence-free.  In addition, $\hat{\bm u}^{n+1}$ 
		retains the same Dirichlet boundary conditions as the original system Eq. \eqref{NS:orig}.} Define $\mathcal{K}(\bu^n,p^n)=\frac12(\int_{\Omega}|\bm u^n|^2dx +\tau^2\Vert\nabla p^n\Vert^2)$, then \eqref{first:prj} bears the following energy law:
	\textcolor{blue}{
		\begin{equation}\label{prj:energy}
			\textstyle    \mathcal{K}(\bu^{n+1},p^{n+1}) - \mathcal{K}(\bu^n,p^n)= -\nu \tau \Vert \nabla \hat{\bm u}^{n+1}\Vert^2-\frac{\Vert\hat\bu^{n+1}-\bu^n\Vert^2}{2}.
	\end{equation}}
	{\color{magenta} 
		Eq. \eqref{prj:energy} can be shown as follows. 
		Taking the $L^2$ inner product on both sides of \eqref{prj:mom:grad} with $\hat{\bm u}^{n+1}$ yields
		\begin{equation*}\label{bdf1:mom:grad:L2:1}
			\textstyle \frac{ \Vert \hat{\bm u}^{n+1} \Vert^2- \Vert \bm u^n\Vert^2 + \Vert \hat{\bm u}^{n+1}-\bm u^n\Vert^2}{2\tau}+(\nabla p^n,\hat{\bm u}^{n+1})+\nu\Vert\nabla\hat{\bm u}^{n+1}\Vert^2=0.
		\end{equation*}
		Testing Eq. \eqref{prj:proj:grad} with $\bu^{n+1}$ and $\nabla p^n$ seperately, there hold
		\begin{equation*}
			\left\{
			\begin{aligned}
				&\Vert \bu^{n+1} \Vert^2 - \Vert \hat{\bu}^{n+1} \Vert^2 + \Vert \bu^{n+1} - \hat{\bu}^{n+1}\Vert^2=0,\\
				&\tau^2 \Vert \nabla p^{n+1}\Vert^2 - \tau^2  \Vert \nabla p^n\Vert^2 - \tau^2\Vert \nabla (p^{n+1}-\nabla p^n)\Vert^2 -2\tau(\hat{\bu}^{n+1}, \nabla p^n)=0,
			\end{aligned}
			\right.
		\end{equation*}
		where we used the fact that $(\nabla p^{n+1},\bu^{n+1})=0$ due to the divergence-free condition \eqref{prj:incom:grad}, boundary conditions of $\bu^{n+1}$, and integration by parts. Moreover, Eq. \eqref{prj:proj:grad} holds
		$$ \tau^2\Vert \nabla (p^{n+1}-\nabla p^n)\Vert^2 = \Vert \bu^{n+1} - \hat{\bu}^{n+1}\Vert^2.$$
	}
	By combining the above 4 equations, we can obtain Eq. \eqref{prj:energy}. As we can see, the pressure correction scheme has a different energy definition from the coupled NS equations \eqref{NS:orig} {\color{blue}discretized with a first-order backward Euler scheme, which is
		\begin{equation*}
			\textstyle 	\frac12 \int_{\Omega}|\bm u^{n+1}|^2dx - \frac12\int_{\Omega}|\bm u^n|^2dx = -\nu \tau \Vert \nabla \bu^{n+1}\Vert^2.
		\end{equation*}
	}
	This is due to the fact that $\nabla\cdot\hat{\bm u}^{n+1}\ne0$, {\color{magenta} so we have to involve the pressure portion into the energy}. 
	Adopting the idea of DRLM schemes, we should incorporate \eqref{prj:energy} or its {\color{magenta}similar} form into the pressure correction scheme. 
	Hence, the first-order pressure-correction DRLM scheme (P-DRLM1) is proposed as follow: given $(\bm u^n, p^n, \mathcal{Q}^n)$, find $(\bm u^{n+1}, p^{n+1}, \mathcal{Q}^{n+1})$ such that 
	\begin{subequations}\label{discretization:1or}
		\begin{empheq}[left=\empheqlbrace]{align}
			&\textstyle\frac{\hat{\bm u}^{n+1}-\bm u^n}{\tau}+\mathcal{Q}^{n+1}\mathcal{N}({\bm u}^n)\bm u^n+\nabla  p^n-\nu\Delta\hat{\bm u}^{n+1}=\boldsymbol{f}^{n+1}, \label{bdf1:mom:grad}\\ 
			&\textstyle\frac{\bm u^{n+1}-\hat{\bm u}^{n+1}}{\tau}+\nabla( p^{n+1}-  p^n)=0,\label{bdf1:proj:grad}\\
			&\textstyle\nabla \cdot \bm u^{n+1}=0,\label{bdf1:incom:grad}\\
			&\textstyle\frac{\mathcal{K}(\bm u^{n+1},p^{n+1})-\mathcal{K}(\bm u^n,p^n)}{\tau}+\theta\frac{(\mathcal{Q}^{n+1})^2-(\mathcal{Q}^n)^2}{\tau}\nonumber \\
			&\textstyle\qquad\qquad\qquad\qquad\qquad\qquad=(\frac{\hat{\bm u}^{n+1}-\bm u^n}{\tau},\hat{\bm u}^{n+1})+\mathcal{Q}^{n+1}(\mathcal{N}({\bm u}^n)\bm u^n,\hat{\bm u}^{n+1})+(\nabla  p^n,\hat{\bm u}^{n+1}),&\label{bdf1:sav:grad}
		\end{empheq}
	\end{subequations}
	where $\mathcal{N}({\bm u})\bv:= (\bm u\cdot\nabla)\bm v$. {\color{blue}The above scheme is   coupled for $(\hat{\bm u}^{n+1}, \bm u^{n+1}, p^{n+1}, \mathcal{Q}^{n+1})$. However, they can be efficiently decoupled, which is fully explained in Section \ref{sec:decouple}.}
	\begin{remark}
		Conventional projection methods \cite{Shen1992SIAM} typically introduce temporal splitting errors that result in a dependence of the pressure $p$ on the time step size $\tau$, specifically $p \sim \mathcal{O}(\tau^{-1})$. If left is unconstrained, this pressure term can severely undermine energy stability for small values of $\tau$, thereby compromising numerical accuracy. The introduced  $\tau^2$ in front of $\Vert\nabla p\Vert^2$ in the energy $\mathcal{K}(\bu,p)$ serves 
		to counteract the $\tau^{-1}$ scaling of the pressure, ensuring that the discrete energy continues to satisfy a dissipation inequality.
		%two principal theoretical purposes:
		%\begin{itemize}
		%    \item to counteract the $\tau^{-1}$ scaling of the pressure, ensuring that the discrete energy continues to satisfy a dissipation inequality;
		%    \item to reinforce the incompressibility constraint within a variational (Galerkin) framework by naturally coupling the velocity and pressure fields through the pressure gradient $\nabla p$, thereby preserving numerical stability and physical consistency. By mitigating this step-size dependency, the correction term guarantees that the energy dissipation inequality
		%\begin{equation*}
		%    E^{n+1}-E^n \leq  - \nu \tau\Vert\nabla\bu^{n+1}\Vert^2
		%\end{equation*}
		%holds even for large time steps.
		%\end{itemize}
		
		%In the context of the incompressible NS equations, neglecting the nonlinear term $\bu \cdot \nabla \bu$ (as in the linear Stokes equations) leads to monotonic decay of kinetic energy, with energy evolution governed solely by viscous dissipation. When the nonlinear term is included, although the global kinetic energy still decays due to viscosity, the nonlinearity enables energy transfer among different vortex structures-such as through turbulent energy cascades-and may induce transient local growth of kinetic energy. Consequently, the design of numerical methods must pay particular attention to ensuring energy stability in the presence of such nonlinear effects.
	\end{remark}

	\subsection{Energy stability}
	Let us state a result on the existence and uniqueness of a strong solution to the NS equations \eqref{NS:orig}, see \cite{Doan2025IMA}.
	\begin{lem}\label{strong solution}
		Assume $\bu_0\in \boldsymbol{V}$ and $\bff\in \boldsymbol{L}^\infty(0,T;\boldsymbol{L}^2(\Omega))$. There exists a positive time $T^*$, with $T^* = T$ if $d = 2$ and $T^* = T^*(\bu_0) \leq T$ if $d = 3$, such that \eqref{NS:orig} admits a unique strong solution $(\bu, p)$ satisfying
		\begin{equation*}
			\begin{aligned}
				\textstyle& \bu\in L^2(0,T^*;\boldsymbol{H}^2(\Omega))\cap \mathscr{C}([0,T^*];\bV),\quad\bu_t\in L^2(0,T^*;\boldsymbol{L}^2(\Omega)), \\
				\textstyle& p\in L^2(0,T^*;H^1(\Omega)/\mathbb{R}).
			\end{aligned}
		\end{equation*}
		Moreover, if $d = 2$ or, in the case $d = 3$, if $\Vert\bu_0\Vert_1$ and $\Vert\bff\Vert_{\boldsymbol{L}^\infty(0,T;\boldsymbol{L}^2(\Omega))}$ are sufficiently small, then the solution $(\bu, p)$ exists for all $t\in [0,T]$, i.e. $T^* = T$ for $d\in{2,3}$, and
		\begin{equation}\label{sup:ut1}
			\sup_{t\in[0,T]}\Vert\bu(t)\Vert_1<\infty.
		\end{equation}
	\end{lem}
	We will demonstrate that the P-DRLM1 scheme \eqref{discretization:1or} is unconditionally energy stable.
	\begin{lem}\label{energy:stab}
		In the absense of the external force $\bff$, the first-order DRLM scheme \eqref{discretization:1or} is unconditionally stable in the sense that
		\begin{equation*}
			\textstyle \mathcal{K}(\bu^{n+1},p^{n+1}) + \theta [(\mathcal{Q}^{n+1})^2-1] \leq  \mathcal{K}(\bm u^n,p^n) + \theta[(\mathcal{Q}^{n})^2-1],n=0,1,\cdots,N.
		\end{equation*}
	\end{lem}
	\begin{proof}
		Taking the $L^2$ inner product on both sides of \eqref{bdf1:mom:grad} with $\hat{\bm u}^{n+1}$ yields
		\begin{equation}\label{bdf1:mom:grad:L2}
			\textstyle(\frac{\hat{\bm u}^{n+1}-\bm u^n}{\tau},\hat{\bm u}^{n+1})+\mathcal{Q}^{n+1}(\mathcal{N}(\bm u^n)\bm u^n,\hat{\bm u}^{n+1})+(\nabla p^n,\hat{\bm u}^{n+1})+\nu\Vert\nabla\hat{\bm u}^{n+1}\Vert^2=0.
		\end{equation}
		Combining Eq. \eqref{bdf1:mom:grad:L2} and \eqref{bdf1:sav:grad} yields
		\begin{equation}\label{energy:law}
			\textstyle\frac{\textcolor{blue}{\mathcal{K}(\bu^{n+1},p^{n+1})-\mathcal{K}(\bm u^n,p^n)}}{\tau}+\theta\frac{(\mathcal{Q}^{n+1})^2-(\mathcal{Q}^n)^2}{\tau}=-\nu\Vert\nabla\hat{\bm u}^{n+1}\Vert^2\leq 0.
		\end{equation}
		This implies that
		\begin{equation*}
			\textstyle \mathcal{K}(\bu^{n+1},p^{n+1}) + \theta (\mathcal{Q}^{n+1})^2 \leq  \mathcal{K}(\bm u^n,p^n) + \theta(\mathcal{Q}^{n})^2 ,
		\end{equation*}
		which completes the proof of Lemma \ref{energy:stab}.
	\end{proof}
	\begin{remark}\label{energy:bound}
		{\color{magenta}
			First, the obtained energy law Eq. \eqref{energy:law} is slightly different from  Eq. \eqref{prj:energy}. This is due to the term $ (\frac{\hat{\bm u}^{n+1}-\bm u^n}{\tau},\hat{\bm u}^{n+1})$  we choose in Eq. \eqref{bdf1:sav:grad} instead of $\frac{\Vert\hat{\bm u}^{n+1} \Vert^2 - \Vert \bm u^n\Vert^2}{2\tau}$.
			This modification ensures P-DRLM1 retains structural similarities with existing DRLM schemes, meanwhile its energy dissipation rate aligns more closely with that of the coupled NS equations.
		}Next, we have from Lemma \ref{energy:stab} that
		\begin{equation*}
			\textstyle0 < \mathcal{K}(\bu^{n+1},p^{n+1}) + \theta (\mathcal{Q}^{n+1})^2 < \mathcal{K}(\bu^n,p^n) + \theta (\mathcal{Q}^n)^2 < \dots < \mathcal{K}(\bu^0,p^0) + \theta (\mathcal{Q}^0)^2 = \mathcal{K}(\bu^0,p^0) + \theta .
		\end{equation*}
		Thus, the discrete energy $\{\mathcal{K}(\bu^n)\}_{n=0}^N$ is uniformly bounded for the first-order scheme \eqref{discretization:1or}.
	\end{remark}
	\begin{remark}\label{2nd:stab}
		
		The second-order pressure-correction DRLM scheme (P-DRLM2) can be designed via the second-order backward differentiation formula, that is: for $n\geq 1$, given $(\bu^n, p^n, \mathcal{Q}^{n})$ and $(\bu^{n-1}, p^{n-1}, \mathcal{Q}^{n-1})$, find $(\bu^{n+1}, p^{n+1}, \mathcal{Q}^{n+1})$, such that
		%, based on the projection method, can be designed as follow
		\begin{subequations}\label{discretization:2or}
			\begin{empheq}[left=\empheqlbrace]{align}
				\textstyle&\frac{3\hat{\bm u}^{n+1}-4\bm u^n+\bu^{n-1}}{2\tau}+\mathcal{Q}^{n+1}\mathcal{N}(\tilde{\bm u}^{n+1})\tilde{\bm u}^{n+1}+\nabla p^n-\nu\Delta\hat{\bm u}^{n+1}=0, \label{bdf2:mom:grad}\\ 
				\textstyle&\frac{3\bm u^{n+1}-3\hat{\bm u}^{n+1}}{2\tau}+\nabla(p^{n+1}-p^n+\nu\nabla\cdot\hat{\bu}^{n+1})=0,\label{bdf2:proj:grad}\\
				\textstyle&\nabla \cdot \bm u^{n+1}=0,\label{bdf2:incom:grad}\\
				\textstyle&\frac{3\mathcal{K}(\bm u^{n+1},p^{n+1})-4\mathcal{K}(\bm u^n,p^n)+\mathcal{K}(\bm u^{n-1},p^{n-1})}{2\tau}+\theta\frac{3(\mathcal{Q}^{n+1})^2-4(\mathcal{Q}^n)^2+(\mathcal{Q}^{n-1})^2}{2\tau}\nonumber \\
				\textstyle&\qquad\qquad\qquad\quad=(\frac{3\hat{\bm u}^{n+1}-4\bm u^n+\bu^{n-1}}{2\tau},\hat{\bm u}^{n+1})+(\mathcal{Q}^{n+1}\mathcal{N}(\tilde{\bm u}^{n+1})\tilde{\bm u}^{n+1}+\nabla p^n,\hat{\bm u}^{n+1}).&\label{bdf2:sav:grad}
			\end{empheq}
		\end{subequations}
		where $\tilde{\bu}^{n+1}=2\bu^n-\bu^{n-1}$. At the first step, we can use P-DRLM1 to compute $(\bu^{1}, p^{1})$ and set $\mathcal{Q}^{1}=1$.
		%In this paper, we will focus on the numerical analysis on the first-order scheme.
		
		%\begin{lem}\label{energy:stab:2or}
		Similarly, when $\bff=0$, P-DRLM2 scheme \eqref{discretization:2or} is unconditionally stable in the sense that
		\begin{equation*}
			\begin{aligned}
				\textstyle &\frac32E(\bu^{n+1},p^{n+1})-\frac12E(\bu^n,p^n) + \theta [\frac32(\mathcal{Q}^{n+1})^2-\frac12(\mathcal{Q}^n)^2-1]\\
				\textstyle\leq&  \frac32E(\bu^n,p^n)-\frac12E(\bm u^{n-1},,p^{n-1}) + \theta[\frac32(\mathcal{Q}^n)^2-\frac12(\mathcal{Q}^{n-1})^2-1],\qquad \forall n=2,\cdots,N.
			\end{aligned}
		\end{equation*}
		%\end{lem}
		The above energy law can be shown by multiplying both sides of equation \eqref{bdf2:mom:grad} with $\hat{\bm u}^{n+1}$, which turns to be
		%\begin{proof}
		%	Multiplying both sides of equation \eqref{bdf2:mom:grad} with $\hat{\bm u}^{n+1}$ , we find that
		\begin{equation*}
			\textstyle(\frac{3\hat{\bm u}^{n+1}-4\bm u^n+\bu^{n-1}}{2\tau},\hat{\bu}^{n+1})+\mathcal{Q}^{n+1}(\mathcal{N}(\tilde{\bm u}^{n+1})\tilde{\bm u}^{n+1},\hat{\bu}^{n+1})+(\nabla p^n,\hat{\bu}^{n+1})+\nu\Vert\nabla\hat{\bm u}^{n+1}\Vert^2=0.
		\end{equation*}
		This, together with \eqref{bdf2:sav:grad}, gives us 
		\begin{equation*}
			\textstyle\frac{3E(\bm u^{n+1},p^{n+1})-4E(\bm u^n,p^n)+E(\bm u^{n-1},p^{n-1})}{2\tau}+\theta\frac{3(\mathcal{Q}^{n+1})^2-4(\mathcal{Q}^n)^2+(\mathcal{Q}^{n-1})^2}{2\tau}\leq 0.
		\end{equation*}
		or equivalently,
		\begin{equation*}
			\textstyle 3E(\bu^{n+1},p^{n+1})-E(\bu^n,p^n) + \theta [3(\mathcal{Q}^{n+1})^2-(\mathcal{Q}^n)^2] \leq  3E(\bu^n,p^n)-E(\bm u^{n-1},p^{n-1}) + \theta[3(\mathcal{Q}^n)^2-(\mathcal{Q}^{n-1})^2] .
		\end{equation*}
		Thus, we obtain the desired estimate.
		In this paper, we will focus on the numerical analysis on P-DRLM1 scheme \eqref{discretization:1or}.
		%	\end{proof}
\end{remark}

\subsection{\textcolor{blue}{Decoupling Strategy and well-posedness of solutions}}
\label{sec:decouple}
We begin by establishing the unique solvability of the Lagrange multiplier $\mathcal{Q}$ and the corresponding regularity properties of the DRLM scheme \eqref{discretization:1or}.
\begin{lem}\label{unique solution}
	The DRLM scheme \eqref{discretization:1or} admits a unique positive solution $\mathcal{Q}^{n+1}>0$ with
	\begin{equation*}
		\bu^n\in\boldsymbol{V}\cap\boldsymbol{H}^2(\Omega),\quad p^n\in H^1(\Omega)/\mathbb{R},\quad \mathcal{Q}^n>0,
	\end{equation*}
	for every $n\in\{0, 1, \cdots, N\}$.
\end{lem}
\begin{proof}
	Follwing \cite{Doan2025JCP,LiShen2020SIAM,Li2022MOC,LIN2019JCP}, the first-order DRLM scheme \eqref{discretization:1or} admits an efficient numerical implementation via a decomposition approach applied to the unknown fields. In particular, by introducing a Lagrange multiplier $\mathcal{Q}^{n+1}$, the velocity field $\bu^{n+1}$, the projected velocity field $\hat{\bu}^{n+1}$,  and the pressure field $p^{n+1}$ are decomposed into two independent components respectively:
	\begin{equation}\label{impl}
		\left\{
		\begin{aligned}
			\textstyle&\hat{\bm u}^{n+1}=\hat{\bm u}_1^{n+1}+\mathcal{Q}^{n+1}\hat{\bu}_2^{n+1}:=\hat{\bm u}_{1}+\mathcal{Q}^{n+1}\hat{\bu}_{2}, \\ 
			\textstyle&\bm u^{n+1}=\bm u_1+\mathcal{Q}^{n+1}\bm u_2,\\
			\textstyle&p^{n+1}=p_1+\mathcal{Q}^{n+1}p_2,
		\end{aligned}
		\right.
	\end{equation}
	This decomposition transforms the original coupled system into a set of decoupled sub-problems, significantly improving computational efficiency while preserving the structure of the numerical scheme. The resulting system is formulated as follows:
	\begin{subequations}\label{decomposition:1or}
		\begin{empheq}[left=\empheqlbrace]{align}
			\textstyle&\frac{\hat{\bu}_1+\mathcal{Q}^{n+1}\hat{\bu}_2-\bu^n}{\tau}+\mathcal{Q}^{n+1}\mathcal{N}(\bm u^n)\bm u^n+\nabla p^n-\nu\Delta(\hat{\bm u}_{1}+\mathcal{Q}^{n+1}\hat{\bu}_{2})=0,\label{decom:mom:grad}\\ 
			\textstyle&\frac{\bu_1+\mathcal{Q}^{n+1}\bu_2-\hat{\bu}_1-\mathcal{Q}^{n+1}\hat{\bu}_2}{\tau}+\nabla(p_1+\mathcal{Q}^{n+1}p_2-p^n)=0,\label{decom:proj:grad}\\
			\textstyle&\nabla\cdot(\bu_1+\mathcal{Q}^{n+1}\bu_2)=0, \label{decom:incom:grad}\\
			\textstyle&\frac{\mathcal{K}(\bu_1+\mathcal{Q}^{n+1}\bu_2,p_1+\mathcal{Q}^{n+1}p_2)-\mathcal{K}(\bm u^n,p^n)}{\tau}+\theta\frac{(\mathcal{Q}^{n+1})^2-(\mathcal{Q}^n)^2}{\tau}\nonumber \\
			\textstyle&\qquad\qquad\qquad\qquad\qquad\qquad\qquad\qquad \qquad\qquad\qquad\qquad\qquad\qquad=-\nu\Vert\nabla(\hat{\bu}_1+\mathcal{Q}^{n+1}\hat{\bu}_2)\Vert^2.&\label{decom:sav:grad}
		\end{empheq}
	\end{subequations}
	The coupled system \eqref{decom:mom:grad} can be decoupled into the following two independent subproblems:
	\begin{subequations}\label{decl}
		\begin{empheq}[left=\empheqlbrace]{align}
			\textstyle &\frac{\hat{\bu}_1-\bu^n}{\tau}+\nabla p^n-\nu\Delta\hat{\bu}_1=0,\label{decl:nonlag}\\
			\textstyle&\frac{\hat{\bu}_2}{\tau}+\mathcal{N}(\bu^n)\bu^n-\nu\Delta\hat{\bu}_2=0,\label{decl:lag}
		\end{empheq}
	\end{subequations}
	We then decouple equations \eqref{decom:proj:grad} and \eqref{decom:incom:grad} into independent generalized Stokes-type systems:
	\begin{subequations}\label{decl:proj:nonlag}
		\begin{empheq}[left=\empheqlbrace]{align}
			\textstyle&\frac{\bu_1-\hat{\bu}_1}{\tau}+\nabla p_1-\nabla p^n=0,\label{proj:nonlag:sys} \\
			\textstyle&\nabla\cdot\bu_1=0,\label{proj:nonlag:div}
		\end{empheq}
	\end{subequations}
	and
	\begin{subequations}\label{decl:proj:lag}
		\begin{empheq}[left=\empheqlbrace]{align}
			\textstyle&\frac{\bu_2-\hat{\bu}_2}{\tau}+\nabla p_2=0,\label{proj:lag:sys} \\
			\textstyle&\nabla\cdot\bu_2=0,\label{proj:lag:div}
		\end{empheq}
	\end{subequations}
	These three systems \eqref{decl}, \eqref{decl:proj:nonlag}, and \eqref{decl:proj:lag} are linear and can be solved independently of $\mathcal{Q}^{n+1}$. Once the variables $\hat{\bu}_1$, $\hat{\bu}_2$, $\bu_1$, $\bu_2$, $p_1$, and $p_2$ are computed, the value of $\mathcal{Q}^{n+1}$ is determined by solving the linear equation derived from \eqref{decom:sav:grad}:
	\begin{equation}\label{impl:bdf1:sav}
		\begin{aligned}
			&\textstyle\frac{\Vert\bu_1\Vert^2+2\mathcal{Q}^{n+1}(\bu_1,\bu_2)+(\mathcal{Q}^{n+1})^2\Vert\bu_2\Vert^2-\Vert\bu^n\Vert^2+\tau^2\Vert\nabla p_1\Vert^2+2\tau^2\mathcal{Q}^{n+1}(\nabla p_1,\nabla p_2)+\tau^2(\mathcal{Q}^{n+1})^2\Vert\nabla p_2\Vert^2-\tau^2\Vert\nabla p^n\Vert^2}{2\tau}\\
			\textstyle&\qquad\qquad\qquad\qquad +\theta\frac{(\mathcal{Q}^{n+1})^2-(\mathcal{Q}^n)^2}{\tau}=-\nu\Vert\nabla\hat{\bu}_1\Vert^2-2\nu\mathcal{Q}^{n+1}(\nabla\hat{\bu}_1,\nabla\hat{\bu}_2)-(\mathcal{Q}^{n+1})^2\nu\Vert\nabla\hat{\bu}_2\Vert^2.
		\end{aligned}
	\end{equation}
	Rearrange and multiplying $2\tau$ on both sides, it holds
	\begin{equation*}
		\begin{aligned}
			\textstyle&(\mathcal{Q}^{n+1})^2\Vert\bu_2\Vert^2+2\mathcal{Q}^{n+1}(\bu_1,\bu_2)+\Vert\bu_1\Vert^2-\Vert\bu^n\Vert^2+2\theta(\mathcal{Q}^{n+1})^2-2\theta(\mathcal{Q}^n)^2+\tau^2\Vert\nabla p_1\Vert^2-\tau^2\Vert\nabla p^n\Vert^2\\
			\textstyle&\qquad\qquad+2\tau^2\mathcal{Q}^{n+1}(\nabla p_1,\nabla p_2)+\tau^2(\mathcal{Q}^{n+1})^2\Vert\nabla p_2\Vert^2+2\nu\tau\Vert\nabla\hat{\bu}_1\Vert^2+4\nu\tau\mathcal{Q}^{n+1}(\nabla\hat{\bu}_1,\nabla\hat{\bu}_2)\\
			\textstyle&\qquad\qquad+2(\mathcal{Q}^{n+1})^2\nu\tau\Vert\nabla\hat{\bu}_2\Vert^2=0,
		\end{aligned}
	\end{equation*}
	It can be further simplified as 
	\vspace{-2ex}
	\begin{equation*}
		\textstyle\mathbf{A}^{n+1}(\mathcal{Q}^{n+1})^2+\mathbf{B}^{n+1}\mathcal{Q}^{n+1}+\mathbf{C}^{n+1}=0,
	\end{equation*}
	where 
	\begin{equation*}
		\left\{
		\begin{aligned}
			\textstyle&\mathbf{A}^{n+1}=\Vert\bu_2\Vert^2+2\theta+\tau^2\Vert\nabla p_2\Vert^2+2\tau\nu \Vert\nabla \hat{\bu}_2\Vert^2,\\ \textstyle&\mathbf{B}^{n+1}=2(\bu_1,\bu_2)+2\tau^2(\nabla p_1,\nabla p_2)+4\nu\tau(\nabla\hat{\bu}_1,\nabla\hat{\bu}_2),\\
			\textstyle&\mathbf{C}^{n+1}=\Vert\bu_1\Vert^2-\Vert\bu^n\Vert^2+\tau^2\Vert\nabla p_1\Vert^2-\tau^2\Vert\nabla p^n\Vert^2-2\theta(\mathcal{Q}^{n})^2+2\tau\nu\Vert\nabla\hat{\bu}_1\Vert^2.
		\end{aligned}
		\right.
	\end{equation*}
	{\color{magenta}
		Taking the inner product of \eqref{decl:nonlag} with $2\tau\hat{\bu}_{1}$, we derive
		\begin{equation*}
			\textstyle	\Vert\hat{\bu}_1\Vert^2-\Vert\bu^n\Vert^2+\Vert\hat{\bu}_1-\bu^n\Vert^2+2\tau(\nabla p^n,\hat{\bu}_1)+2\nu\tau\Vert\nabla\hat{\bu}_1\Vert^2=0.
		\end{equation*}
	By substituding the equation into $\mathbf{C}^{n+1}$, it yields
	\begin{equation*}
		\textstyle\mathbf{C}^{n+1}=\textstyle\Vert\bu_1\Vert^2-\Vert\hat{\bu}_1\Vert^2-\Vert\hat{\bu}_1-\bu^n\Vert^2-2\tau(\nabla p^n,\hat{\bu}_1)+\tau^2\Vert\nabla p_1\Vert^2-\tau^2\Vert\nabla p^n\Vert^2-2\theta(\mathcal{Q}^{n})^2.
	\end{equation*}
		Next, taking the inner product of the first equation in \eqref{decl:proj:nonlag} with $\hat{\bu}_1$ and $\nabla p_1$, respectively, we obtain
		\begin{equation*}
			\left\{
			\begin{aligned}
				\Vert\bu_1\Vert^2-\Vert\hat{\bu}_1\Vert^2-2\tau(\nabla p^n,\hat{\bu}_1)&=\Vert\bu_1-\hat{\bu}_1\Vert^2-2\tau(\nabla p_1,\hat{\bu}_1),\\
				-2\tau(\hat{\bu}_1,\nabla p_1)+\tau^2\Vert\nabla p_1\Vert^2-\tau^2\Vert\nabla p^n\Vert^2&=-\tau^2\Vert\nabla(p_1-p^n)\Vert^2.
			\end{aligned}
			\right.
		\end{equation*}
		Then, $ \mathbf{C}^{n+1} = \textstyle
		\Vert\bu_1-\hat{\bu}_1\Vert^2-\Vert\hat{\bu}_1-\bu^n\Vert^2-\tau^2\Vert\nabla(p_1-p^n)\Vert^2-2\theta(\mathcal{Q}^{n})^2$.
		Moving the second and third terms of \eqref{proj:nonlag:sys} to the right-hand side and squaring both sides yields
		\begin{equation*}
			\Vert\bu_1-\hat{\bu}_1\Vert^2=\tau^2\Vert\nabla(p_1-p^n)\Vert^2.
		\end{equation*}
		Eventually, $\textstyle\mathbf{C}^{n+1}= \textstyle -\Vert\hat{\bu}_1-\bu^n\Vert^2 -2\theta(\mathcal{Q}^{n})^2<0.$
	}
	We remark that $\mathbf{A}^{n+1}>0$ and $\mathbf{C}^{n+1}<0$ for any $ \theta > 0 $. Consequently, the quadratic equation in $\mathcal{Q}$ admits a unique positive solution for any positive regularization parameter $\theta$. Once $\mathcal{Q}^{n+1}$ is determined, the updated values $\bu^{n+1}$ and $p^{n+1}$ can be recovered explicitly through the relations given in \eqref{impl}.
\end{proof}
{\color{magenta}
	\vspace{-2ex}
	\begin{remark}
		If we choose $\frac{\Vert\hat{\bm u}^{n+1} \Vert^2 - \Vert \bm u^n\Vert^2}{2\tau}$ instead of $ (\frac{\hat{\bm u}^{n+1}-\bm u^n}{\tau},\hat{\bm u}^{n+1})$in Eq. \eqref{bdf1:sav:grad}, the coefficients $A$ and $C$ are
		\begin{equation*}
			\left\{
			\begin{aligned}
				\textstyle&\mathbf{A}^{n+1}=\Vert\bu_2\Vert^2+\Vert \hat{\bu}_2 \Vert^2+2\theta+\tau^2\Vert\nabla p_2\Vert^2+2\tau\nu \Vert\nabla \hat{\bu}_2\Vert^2>0,\\ 
				\textstyle&\mathbf{C}^{n+1}=~-2\theta(\mathcal{Q}^{n})^2<0.
			\end{aligned}
			\right.
		\end{equation*}
		Hence, $ \mathcal{Q}^{n+1}$ is still unconditionally solvable.
	\end{remark}
}
Next, we show that the Lagrange multiplier in the DRLM scheme \eqref{discretization:1or} is uniformly bounded for any time step size $\tau > 0$.
\begin{lem}\label{Q:bound}
	Let $\left\{\mathcal{Q}^n\right\}_{n=0}^N$ be the positive sequence generated by the DRLM scheme \eqref{discretization:1or}. There exists a constant $C_0\geq1$ and $\tilde{M}$ , depending only on $T,~ \Vert\bu^0\Vert$ , such that $\mathcal{Q}^n\leq C_0,~ \Vert\bu^n\Vert\leq\tilde{M},~ \forall n=0,1,\ldots,N$.
\end{lem}
\begin{proof}
	Taking the $L^2$ inner product of \eqref{bdf1:mom:grad} with $\hat{\bu}^{n+1}$ yields
	\begin{equation}\label{inner:bdf1:mom}
		\begin{aligned}
			&\textstyle(\frac{\hat{\bm u}^{n+1}-\bm u^n}{\tau},\hat{\bu}^{n+1})+\mathcal{Q}^{n+1}(\mathcal{N}(\bm u^n)\bm u^n,\hat{\bu}^{n+1})+\mathcal{Q}^{n+1}(\nabla p^n,\hat{\bu}^{n+1})+\nu\Vert\hat{\bu}^{n+1}\Vert^2=0,
		\end{aligned}
	\end{equation}
	From \eqref{bdf1:sav:grad} and \eqref{inner:bdf1:mom}, we obtain
	\begin{equation}\label{ene:sup:bound}
		\textstyle\frac{\Vert\bu^{n+1}\Vert^2-\Vert\bu^n\Vert^2+\tau^2\Vert\nabla p^{n+1}\Vert^2-\tau^2\Vert\nabla p\Vert^2}{2\tau}+\theta\frac{(\mathcal{Q}^{n+1})^2-(\mathcal{Q}^n)^2}{\tau}=-\nu\Vert\hat{\bu}^{n+1}\Vert^2\leq0,
	\end{equation}
	Summing \eqref{ene:sup:bound} over $n=0, 1, \cdots, N$, noting that $\mathcal{Q}^0=1$, giving us
	\begin{equation*}
		\begin{aligned}
			\textstyle&\mathcal{Q}^{n+1}\leq\sqrt{1+\frac{\Vert\bu^0\Vert^2+\tau^2\Vert\nabla p^0\Vert^2}{2\theta}}= C_0,\\
			\textstyle&\Vert\bu^{n+1}\Vert\leq\sqrt{\Vert\bu^0\Vert^2+\tau^2\Vert\nabla p^0\Vert^2+2\theta}=\tilde{M}.
		\end{aligned}
	\end{equation*}
	The proof of the lemma is thus complete.
\end{proof}
%%%%%%%%%%%%%%%%%%%%%%%%%%%%%%%%%%%%%%%%%%
\section{Temporal error estimate}\label{errest}
%%%%%%%%%%%%%%%%%%%%%%%%%%%%%%%%%%%%%%%%%%
Let $ {\bu^n, p^n, \mathcal{Q}^n}$ be the numerical solution at time $t_n$ obtained by  P-DRLM1 scheme \eqref{discretization:1or}, we define the errors of the velocity, pressure, and Lagrange multiplier as follows:
\begin{equation*}
	e_{\bu}^n=\bu(t_n)-\bu^n,\quad e_p^n=p(t_n)-p^n,\quad e_Q^n=1-\mathcal{Q}^n.
\end{equation*}
In particular, we have $e_{\bu}^0=0$ and $e_Q^0=0$. Before stating the main results of this work, let us first recall some useful inequalities.
\subsection{Preliminaries}
Due to the imposed homogeneous Dirichlet boundary conditions, the velocity $\boldsymbol{u}\in \boldsymbol{H}_0^1(\Omega)$. For the trace-free $\boldsymbol{H}^1$ functions, we have the 
$\text{Poincar\'{e}}$ inequality \cite{evans2022partial}, that is, for some positive constant $C_1$ depending only on the domain $\Omega$, there exists
\begin{equation}
	\Vert\boldsymbol{v}\Vert_1\leq(1+C_1)\Vert\nabla\boldsymbol{v}\Vert,\quad\forall\boldsymbol{v}\in\boldsymbol{H}_0^1(\Omega).
\end{equation}

For $\bu,\bv,\bw\in H_0^1(\Omega)$, we define the trilinear form $b(\cdot,\cdot,\cdot)$ by
\begin{equation*}
	b(\bu,\bv,\bw)=\int_\Omega\left((\bu\cdot\nabla)\bv\right)\cdot \bw dx.
\end{equation*}
It can be shown that $b(\cdot,\cdot,\cdot)$ is skew-symmetric with respect to its last two arguments, i.e.,
\begin{equation}\label{skew-symmetric}
	b(\boldsymbol{u},\boldsymbol{v},\boldsymbol{w})=-b(\boldsymbol{u},\boldsymbol{w},\boldsymbol{v}),\quad\forall\boldsymbol{u}\in\boldsymbol{V},\forall\boldsymbol{v},\boldsymbol{w}\in\boldsymbol{H}_0^1(\Omega).
\end{equation}
Hence, by setting $\bw = \bv$ in \eqref{skew-symmetric}, we have $b(\bu,\bv,\bv)=0,\forall\bv\in\boldsymbol{V},\forall\bv\in\boldsymbol{H}_0^1(\Omega)$. Furthermore, for $d \leq 4$, there exists a constant $C_2$ depending on $\Omega$ such that
\begin{equation}\label{skew-C2}
	|b(\boldsymbol{u},\boldsymbol{v},\boldsymbol{w})|\leq C_{2}
	\begin{cases}
		\|\boldsymbol{u}\|_{1}\|\boldsymbol{v}\|_{1}\|\boldsymbol{w}\|_{1}, \\
		\|\boldsymbol{u}\|_{0}\|\boldsymbol{v}\|_{2}\|\boldsymbol{w}\|_{1}, \\
		\|\boldsymbol{u}\|_{1}\|\boldsymbol{v}\|_{2}\|\boldsymbol{w}\|_{0}, \\
		\|\boldsymbol{u}\|_{0}\|\boldsymbol{v}\|_{1}\|\boldsymbol{w}\|_{2}, \\
		\|\boldsymbol{u}\|_{2}\|\boldsymbol{v}\|_{1}\|\boldsymbol{w}\|_{0}. &  
	\end{cases}
\end{equation}
The above estimates are widely utilized and can be obtained by the H\"older's inequality and Sobolev embedding theorems \cite{evans2022partial}. The detailed proof can be found in \cite{He2003,HeLiu2005,HeywoodRannacher1982}.

Lastly, we need the following discrete version of the Gr\"onwall inequality (see, for example, \cite{Diegel2017,HeSun2007,QuarteroniValli2008}) to deal with the temporal discretization. Its content is stated below.
\begin{lem}\label{discrete Gronwall}
	Let $\{a_n\}_{n=0}^K, \{b_n\}_{n=0}^K, \{c_n\}_{n=0}^K$ be nonnegative sequences with $\omega_0a_0+c_1\leq c_1e^{\omega_0}$. If ${c_n}$ is non-decreasing and
	\begin{equation*}
		a_n+b_n\leq c_n+\sum_{i=0}^{n-1}\omega_ia_i,\quad\forall\mathrm{~}1\leq n\leq K,
	\end{equation*}
	then we have
	\begin{equation}\label{gronwall:exp}
		a_{n}+b_{n}\leq c_{n}\exp\left(\sum_{i=0}^{n-1}\omega_{i}\right),\quad\forall\mathrm{~}1\leq n\leq K.
	\end{equation}
\end{lem}
%The condition $\omega_0a_0 +c_1 \leq c_1e^{\omega_0}$ in Lemma \ref{discrete Gronwall} is needed to ensure \eqref{gronwall:exp} holds for $n = 1$. This condition is satisfied if $a_0 \leq c_1$.
\subsection{Optimal error estimates for the velocity and Lagrange multiplier}
For the subsequent analysis, we require the solution $(\bu, p)$ of the NS equations \eqref{NS:orig} to satisfy the following regularity conditions, which are also required in \cite{Li2022MOC}:
\begin{equation}\label{regu}
	\begin{aligned}
		\textstyle& \bu\in H^3(0,T;\boldsymbol{L}^2(\Omega))\cap H^1(0,T;\boldsymbol{H}^2(\Omega))\cap W^{1,\infty}(0,T;\boldsymbol{W}^{1,\infty}(\Omega)), \\
		\textstyle& p\in H^2(0,T;\boldsymbol{H}^1(\Omega)/\mathbb{R}).
	\end{aligned}
\end{equation}
Let $M$ be a constant such that
\begin{equation}\label{maxM}
	\max\left\{\sup_{t\in[0,T]}\|\boldsymbol{u}\|_2,\mathrm{~}\int_0^T\|\boldsymbol{u}_t\|_1^2dt,\mathrm{~}\int_0^T\|\boldsymbol{u}_{tt}\|_{-1}^2dt,\tilde{M}\right\}\leq M.
\end{equation}
\begin{lem}\label{velocity estimate1}
	Let $\{\bu^1\}$, $\{p^1\}$, and $\{Q^1\}$ be generated by the first-order DRLM scheme \eqref{discretization:1or}. There exists $\tau_0>0$. Under the regularity assumptions \eqref{regu}, there exist positive constants $\tau$, and $C^\star$ depending on $\Omega$, $T$, $\theta$, $\nu$, $M$ and $\bu^0$ but independent of $\tau$ and $n$ such that the following error estimates hold for all $\tau < \tau_0$:
	\begin{equation}\label{u:Q:err:1}
		\textstyle\Vert e_{\bu}^{1}\Vert^2+|e_Q^{1}|^2+\tau^2\Vert\nabla e_p^{1}\Vert^2+\nu\tau\Vert\nabla\hat{e}_{\bu}^{1}\Vert^2\leq C^\star\tau^2.
	\end{equation}
\end{lem}
\begin{proof}
	To derive optimal error estimates for both the velocity and the Lagrange multiplier at $t=t_1$, we proceed through the following steps.
	
	\textbf{Step 1: Estimates of velocity. } Subtracting the momentum equation \eqref{NS:orig:mom} at $t_1$ from \eqref{bdf1:mom:grad}, we derive the error equations for velocity:
	\begin{subequations}\label{error:proj:1}
		\begin{empheq}[left=\empheqlbrace]{align}
			&\textstyle\frac{\hat{e}_{\bu}^{1}-e_{\bu}^0}{\tau}-\nu\Delta\hat{e}_{\bu}^{1}=-\mathcal{Q}(t_{1})\mathcal{N}(\bu(t_{1}))\bu(t_{1})+\mathcal{Q}^{1}\mathcal{N}(\bm u^0)\bm u^0\nonumber\\
			&\textstyle\qquad\qquad\qquad\qquad\qquad-\nabla(p(t_{1})-p^0)-\boldsymbol{R}_{\bu}^{1}, \label{err:mom:grad:1}\\ 
			&\textstyle\frac{e_{\bu}^{1}-\hat{e}_{\bu}^{1}}{\tau}-\nabla(p^{1}-p^0)=0,\label{err:proj:grad:1}
		\end{empheq} 
	\end{subequations}
	where the truncation error $\bR_{u}^{1}$ is given by
	\begin{equation}\label{trunc:velo:1}
		\bR_{u}^{1} = \bu_{t}(t_{1})-\frac{\bu(t_{1})-\bu(t_0)}{\tau}=\frac{1}{\tau}\int_{t_{0}}^{t_{1}} (t-t_0)u_{tt}(t) \, dt.
	\end{equation}
	Taking the inner product of \eqref{err:mom:grad:1} with $\hat{e}_{\bu}^{1}$, we obtain
	\begin{equation}\label{err:mom:L2:1}
		\begin{aligned}
			&\frac{\Vert\hat{e}_{\bu}^{1}\Vert^2-\Vert e_{\bu}^0\Vert^2}{2\tau}+\frac{\Vert\hat{e}_{\bu}^{1}-e_{\bu}^0\Vert^2}{2\tau}+\nu\Vert\nabla\hat{e}_{\bu}^{1}\Vert^2\\
			\textstyle=&-\mathcal{Q}(t_{1})b(\bu(t_{1}),\bu(t_{1}),\hat{e}_{\bu}^{1})+\mathcal{Q}^{1}b(\bm u^0,\bm u^0,\hat{e}_{\bu}^{1})-(\nabla(p(t_{1})-p^0),\hat{e}_{\bu}^{1})-(\boldsymbol{R}_{\bu}^{1},\hat{e}_{\bu}^{1}),
		\end{aligned}
	\end{equation}
	Taking the inner product of \eqref{err:proj:grad:1} with $\frac{e_{\bu}^{1}+\hat{e}_{\bu}^{1}}{2}$, we derive
	\begin{equation}\label{err:proj:L2:1}
		\textstyle\frac{\Vert e_{\bu}^{1}\Vert^2-\Vert\hat{e}_{\bu}^{1}\Vert^2}{2\tau}-\frac{1}{2}(\nabla(p^{1}-p^0),\hat{e}_{\bu}^{1})=0
	\end{equation}
	Adding \eqref{err:mom:L2:1} and \eqref{err:proj:L2:1} with $e_{\bu}^0=0$, we have
	\begin{equation}\label{err:velo:add:1}
		\begin{aligned}
			&\frac{\Vert e_{\bu}^{1}\Vert^2+\Vert\hat{e}_{\bu}^{1}\Vert^2}{2\tau}+\nu\Vert\nabla\hat{e}_{\bu}^{1}\Vert^2\\
			\textstyle=&-(\mathcal{Q}(t_{1})(\bu(t_{1})\cdot\nabla)\bu(t_{1})-\mathcal{Q}^{1}(\bm u^0\cdot\nabla)\bm u^0,\hat{e}_{\bu}^{1})-\frac{1}{2}(\nabla(2p(t_{1})-p^{1}-p^0),\hat{e}_{\bu}^{1})\\
			&\textstyle-(\boldsymbol{R}_{\bu}^{1},\hat{e}_{\bu}^{1}).
		\end{aligned}
	\end{equation}
	For the first term on the right hand side of \eqref{err:velo:add:1}, it follows from properties of $b(\cdot,\cdot,\cdot)$ in \eqref{skew-C2} and the regularity assumption, so that
	\begin{equation}\label{nonlinear:1}
		\begin{aligned}
			\textstyle&-(\mathcal{Q}(t_{1})(\bu(t_{1})\cdot\nabla)\bu(t_{1})-\mathcal{Q}^{1}(\bm u^0\cdot\nabla)\bm u^0,\hat{e}_{\bu}^{1})\\
			\textstyle=& ~-b(\bu(t_{1})-\bu(t_0),\bu(t_{1}),\hat{e}_{\bu}^{1})-b(\bu(t_0),\bu(t_{1})-\bu(t_0),\hat{e}_{\bu}^{1})-e_{Q}^{1}b(\bu(t_0),\bu(t_0),\hat{e}_{\bu}^{1})\\
			\textstyle\leq& ~ 2(1+C_1)C_2\Vert\bu(t_{1})-\bu(t_0)\Vert\Vert\nabla\hat{e}_{\bu}^{1}\Vert+C_2(1+C_1)M^2|e_Q^{1}|\Vert\nabla\hat{e}_{\bu}^{1}\Vert\\
			\textstyle\leq& \textstyle~\frac{64(1+C_1)^2C_2^2}{\nu}\tau\int_{t_0}^{t_{1}}\Vert\bu_t\Vert^2dt+\frac{\nu}{32}\Vert\nabla\hat{e}_{\bu}^{1}\Vert^2+\frac{16}{\nu}C_2^2(1+C_1)^2M^4|e_Q^{1}|^2.
		\end{aligned}
	\end{equation}
	With \eqref{err:proj:grad:1} and $e_p^0=0$, the third and forth terms on the right-hand side of \eqref{err:velo:add:1} can be bounded by
	\begin{equation}\label{err:p:1}
		\begin{aligned}
			\textstyle-\frac{1}{2}(\nabla(2p(t_{1})-p^{1}-p^0),\hat{e}_{\bu}^{1})=&~\frac{\tau}{2}(\nabla(e_p^{1}+e_p^0+p(t_{1})-p(t_0)),\nabla(-e_p^{1}+e_p^0+p(t_{1})-p(t_0)))\\
			\textstyle=&-\frac{\tau}{2}\Vert\nabla e_p^{1}\Vert^2+\frac{\tau}{2}\Vert\nabla(p(t_{1})-p(t_0))\Vert^2\\
			\textstyle\leq&-\frac{\tau}{2}\Vert\nabla e_p^{1}\Vert^2+\frac{\tau^2}{2}\int_{t_0}^{t_{1}}\Vert\nabla p_t\Vert^2dt,
		\end{aligned}
	\end{equation}
	and
	\begin{equation}\label{err:trunc:1}
		\begin{aligned}
			\textstyle -(\boldsymbol{R}_{\bu}^{1},\hat{e}_{\bu}^{1})\leq& \textstyle~
			~\left\Vert\frac{1}{\tau}\int_{t_{0}}^{t_{1}}(t-t_{0})\bu_{tt}dt\right\Vert_{-1}\Vert e_{\bu}^{1}\Vert_{1}\\
			\leq&\textstyle~(1+C_1)\left(\int_{t_{0}}^{t_{1}}\Vert\bu_{tt}\Vert_{-1}dt\right)\Vert\nabla \hat{e}_{\bu}^{1}\Vert\\
			\textstyle\leq&\textstyle
			~\frac{8(1+C_1)^2}{\nu}\tau\int_{t_{0}}^{t_{1}}\Vert\bu_{tt}\Vert_{-1}^2dt+\frac{\nu}{32}\Vert\nabla\hat{e}_{\bu}^{1}\Vert^2.
		\end{aligned}
	\end{equation}
	Combine \eqref{err:velo:add:1}-\eqref{err:trunc:1} and multiply the resulting estimate by $2\tau$, it yields
	\begin{equation}\label{err:velo:sum:1}
		\begin{aligned}
			\textstyle&\Vert e_{\bu}^{1}\Vert^2+\Vert\hat{e}_{\bu}^{1}\Vert^2+\tau^2\Vert\nabla e_p^{1}\Vert^2+\frac{15\nu}{8}\Vert\nabla\hat{e}_{\bu}^{1}\Vert^2\\
			\textstyle\leq&~\frac{128(1+C_1)^2C_2^2}{\nu}\tau^2\int_{t_0}^{t_{1}}\Vert\bu_t\Vert^2dt+\frac{32}{\nu}C_2^2(1+C_1)^2M^4\tau|e_Q^{1}|^2+\frac{16(1+C_1)^2}{\nu}\tau^2\int_{t_{n}}^{t_{n+1}}\Vert\bu_{tt}\Vert_{-1}^2dt\\
			\textstyle&+\tau^3\int_{t_n}^{t_{n+1}}\Vert\nabla p_t\Vert^2dt.
		\end{aligned}
	\end{equation}
	\textbf{Step 2: Estimates of the Lagrange multiplier.}  The presence of the Lagrange multiplier error term $|e_Q^{1}|^2$ on the right-hand side of \eqref{err:velo:sum:1} requires an estimate. To do so, we first utilize the quantity  $|\mathcal{Q}^{1}+1|\geq 1$, recalling that the exact value of $\mathcal{Q}(t)$ is 1 and $\mathcal{Q}^0$= 1. From equation \eqref{bdf1:proj:grad}, we obtain the relation $\Vert\bu^1-\hat{\bu}^1\Vert^2=\tau^2\Vert\nabla(p^1-p^0)\Vert^2$. Taking $L^2$ inner product of \eqref{bdf1:proj:grad} with $\nabla p^0$ and adding \eqref{bdf1:sav:grad}, it holds:
	\begin{equation}\label{Q:err:orig:1}
		\begin{aligned}
			\textstyle&\theta\frac{(\mathcal{Q}^{1})^2-(\mathcal{Q}^0)^2}{\tau}\\
			\textstyle=&~\frac{\Vert\bu^{1}-\hat{\bu}^{1}\Vert^2+\Vert\hat{\bu}^{1}-\bu^0\Vert^2}{2\tau}+\mathcal{Q}^{1}b(\bu^0,\bu^0,\hat{\bu}^{1})+(\nabla p^0,\hat{\bu}^{1})-\frac{\tau}{2}(\Vert\nabla p^{1}\Vert^2-\Vert\nabla p^0\Vert^2)\\
			\textstyle=&~\frac{\tau^2\Vert\nabla(p^{1}-p^0)\Vert^2+\Vert\hat{\bu}^1-\bu^0\Vert^2}{2\tau}+\mathcal{Q}^{1}b(\bu^0,\bu^0,\hat{\bu}^{1})+(\nabla p^0,\hat{\bu}^{1})-\frac{\tau}{2}(\Vert\nabla p^{1}\Vert^2-\Vert\nabla p^0\Vert^2)\\
			\textstyle=&~\frac{\Vert\hat{\bu}^{1}-\bu^0\Vert^2}{2\tau}+\mathcal{Q}^{1}b(\bu^0,\bu^0,\hat{\bu}^{1}).
		\end{aligned}
	\end{equation}
	Noting that $\mathcal{Q}^0=1$ and $\mathcal{Q}^1>0$, then
	\begin{equation}\label{Q:err:mid:1}
		\begin{aligned}
			\textstyle\theta\frac{e_Q^{1}-e_Q^0}{\tau}=&~\frac{1}{\mathcal{Q}^{1}+1}\left(\frac{\Vert\hat{\bu}^{1}-\bu^0\Vert^2}{2\tau}+\mathcal{Q}^{1}b(\bu^0,\bu^0,\hat{\bu}^{1})\right)
			%\textstyle\leq&~\frac{\Vert\hat{e}_{\bu}^{1}\Vert^2+\Vert\bu(t_{1})-\bu(t_0)\Vert^2}{\tau}+|\mathcal{Q}^{1}b(\bu^0,\bu^0,\hat{\bu}^{1})|.
		\end{aligned}
	\end{equation}
	Taking $L^2$ inner product of \eqref{Q:err:mid:1} with $e_Q^{1}$, we obtain
	\begin{equation}\label{Q:err:1}
		\begin{aligned}
			\textstyle\theta\frac{|e_Q^{1}|^2}{\tau}
			\leq&~ \frac{|e_Q^{1}|}{\mathcal{Q}^{1}+1 }\left( \frac{\Vert\hat{\bu}^{1}-\bu^0\Vert^2}{2\tau}+|\mathcal{Q}^{1}b(\bu^0,\bu^0,\hat{\bu}^{1})|\right)
			\\
			\leq&~|e_Q^{1}|\left(\frac{\Vert\hat{e}_{\bu}^{1}\Vert^2+\Vert\bu(t_{1})-\bu(t_0)\Vert^2}{\tau}+|\mathcal{Q}^{1}b(\bu^0,\bu^0,\hat{\bu}^{1}) |\right)\\
			\textstyle\leq&~|e_Q^{1}|\frac{\Vert\hat{e}_{\bu}^{1}\Vert^2+\Vert\bu(t_{1})-\bu(t_0)\Vert^2}{\tau}+C_0|e_Q^{1}b(\bu^0,\bu^0,\hat{\bu}^{1})|\\
			\textstyle\leq&~|e_Q^{1}|\frac{\Vert\hat{e}_{\bu}^{1}\Vert^2}{\tau}+|e_Q^{1}|^2+\frac14\tau\int_{t_0}^{t_{1}}\Vert\bu_t\Vert^4dt+C_0|e_Q^{1}b(\bu^0,\bu^0,\hat{\bu}^{1})|.
		\end{aligned}
	\end{equation}
	Using the properties of $b(\cdot,\cdot,\cdot)$ (cf.\eqref{skew-C2}), we find
	\begin{equation}\label{Q:b:1}
		\begin{aligned}
			\textstyle C_0|e_Q^{1}b(\bu^0,\bu^0,\hat{\bu}^{1})|=&~C_0|e_Q^{1}b(\bu(t_0),\bu(t_0),\bu(t_{1})-\hat{e}_{\bu}^{1})|\\
			\textstyle\leq&~C_0|e_Q^{1}\left(b(\bu(t_0),\bu(t_0),\bu(t_{1})-\bu(t_0))-b(\bu(t_0),\bu(t_0),\hat{e}_{\bu}^{1})\right)|\\
			\leq&\textstyle
			~ C_0^2C_2^2M^4\left(1+\frac{2(1+C_1)^2}{\theta\nu}\right)|e_Q^{1}|^2+\frac14\tau\int_{t_0}^{t_{1}}\Vert\bu_t\Vert^2dt +\frac{\theta\nu}{8}\Vert\nabla\hat{e}_{\bu}^{1}\Vert^2.
		\end{aligned}
	\end{equation} 
	Observing that $\mathcal{Q}^1\leq C_0$ implies $|e_Q^1|\leq 1+C_0$, so we choose $\theta\geq 2(1+C_0)$. Combining \eqref{Q:b:1}  with \eqref{Q:err:1}, we derive
	\begin{equation*}\label{Q:error:1}
		\textstyle\frac{|e_Q^{1}|^2}{\tau}\leq C_3^\star|e_Q^{1}|^2+\frac{\Vert\hat{e}_{\bu}^{1}\Vert^2}{2\tau}+\frac{\tau}{4\theta}\int_{t_0}^{t_{1}}\Vert\bu_t\Vert^4dt+\frac{\tau}{4\theta}\int_{t_0}^{t_{1}}\Vert\bu_t\Vert^2dt +\frac{\nu}{8}\Vert\nabla\hat{e}_{\bu}^{1}\Vert^2,
	\end{equation*}
	where 
	\begin{equation*}
		\textstyle C_3^\star=\frac{C_0^2C_2^2M^4}{\theta}\left(1+\frac{2(1+C_1)^2}{\theta\nu}\right)+\frac{1}{\theta}.
	\end{equation*}
	Multiplying $\tau$ on the both sides, we obtain
	\begin{equation}\label{Q:error:sum:1}
		\textstyle|e_Q^{1}|^2\leq C_3^\star\tau|e_Q^{1}|^2+\frac{\Vert\hat{e}_{\bu}^{1}\Vert^2}{2}+\frac{\tau^2}{4\theta}\int_{t_0}^{t_{1}}\Vert\bu_t\Vert^4dt+\frac{\tau^2}{4\theta}\int_{t_0}^{t_{1}}\Vert\bu_t\Vert^2dt +\frac{\nu}{8}\tau\Vert\nabla\hat{e}_{\bu}^{1}\Vert^2.
	\end{equation}
	\textbf{Step 3: Simultaneous estimates for the velocity and Lagrange multiplier.} In light of Step 1 and Step 2,  we find
	\begin{equation}\label{velo:lagrange:err:1}
		\begin{aligned}
			\textstyle&\Vert e_{\bu}^{1}\Vert^2+\frac12\Vert\hat{e}_{\bu}^{1}\Vert^2+\tau^2\Vert\nabla e_p^{1}\Vert^2+\frac{7}{4}\nu\tau\Vert\nabla\hat{e}_{\bu}^{1}\Vert^2+|e_Q^{1}|^2\\
			\leq&\textstyle
			~\frac{128(1+C_1)^2C_2^2}{\nu}\tau^2\int_{t_0}^{t_{1}}\Vert\bu_t\Vert^2dt+(\frac{32}{\nu}C_2^2(1+C_1)^2M^4+C_3^\star)\tau|e_Q^{1}|^2+\frac{16(1+C_1)^2}{\nu}\tau^2\int_{t_{0}}^{t_{1}}\Vert\bu_{tt}\Vert_{-1}^2dt\\
			\textstyle&\textstyle
			+\tau^3\int_{t_0}^{t_{1}}\Vert\nabla p_t\Vert^2dt+\frac{\tau^2}{2\theta}\int_{t_0}^{t_{1}}\Vert\bu_t\Vert^4dt+\frac{\tau^2}{4\theta}\int_{t_0}^{t_{1}}\Vert\bu_t\Vert^2dt .
		\end{aligned}
	\end{equation}
	Let $\tau\leq(\frac{64}{\nu}C_2^2(1+C_1)^2M^4+2C_3^\star)^{-1}$, we derive
	\begin{equation}\label{err:1}
		\textstyle\Vert e_{\bu}^{1}\Vert^2+\Vert\hat{e}_{\bu}^{1}\Vert^2+\tau^2\Vert\nabla e_p^{1}\Vert^2+\frac{7}{4}\nu\tau\Vert\nabla\hat{e}_{\bu}^{1}\Vert^2+\frac12|e_Q^{1}|^2\leq C^\star\tau^2.
	\end{equation}
	where
	\begin{equation*}
		\begin{aligned}
			\textstyle C^\star=&
			\textstyle
			\frac{128(1+C_1)^2C_2^2}{\nu}\int_{t_0}^{t_{1}}\Vert\bu_t\Vert^2dt+\frac{16(1+C_1)^2}{\nu}\int_{t_{0}}^{t_{1}}\Vert\bu_{tt}\Vert_{-1}^2dt\\
			\textstyle&
			\textstyle
			+\tau\int_{t_0}^{t_{1}}\Vert\nabla p_t\Vert^2dt+\frac{1}{2\theta}\int_{t_0}^{t_{1}}\Vert\bu_t\Vert^4dt+\frac{1}{4\theta}\int_{t_0}^{t_{1}}\Vert\bu_t\Vert^2dt.
		\end{aligned}
	\end{equation*}
\end{proof}
Thanks to Lemma \ref{velocity estimate1} and mathematical induction, we now can prove the following theorem.
\begin{thm}\label{velocity estimate}
	Let $\{\bu^i\}_{i=1}^n$, $\{p^1i\}_{i=1}^n$, and $\{Q^i\}_{i=1}^n$ be generated by the P-DRLM1 scheme \eqref{discretization:1or}. Under the regularity assumptions \eqref{regu}, there exist positive constants $\tau$, $C_1$, and $C_2$ depending on $\Omega$, $T$, $\nu$, $M$ and $\bu^0$ but independent of $\tau$ and $n$ such that the following error estimates hold for all $\tau < \tau_0:=\min\{\frac{1}{(\frac{C_6C_8}{\nu}+ C_9)^2}, \frac{C_6C_8+C_9}{2}\}, \theta\geq 8(1+C_0)$:
	\begin{equation}\label{u:Q:err}
		\textstyle\Vert e_{\bu}^{n+1}\Vert^2+\sum\limits_{i=0}^n\Vert\hat{e}_{\bu}^{i+1}-e_{\bu}^i\Vert^2+|e_Q^{n+1}|^2+\tau^2\Vert\nabla e_p^{n+1}\Vert^2+\nu\tau\sum\limits_{i=0}^n\Vert\nabla\hat{e}_{\bu}^{i+1}\Vert^2\leq (C_6C_8+C_9)\tau^2,\quad0\leq n\leq N-1.
	\end{equation}
\end{thm}
\begin{proof}
	This proof proceeds by mathematical induction. According to Lemma \ref{velocity estimate1}, the estimate \eqref{u:Q:err} holds for $t = t_1$. Now assume that \eqref{u:Q:err} is valid for all $t = t_1, \dots, t_n$. We aim to show that it also holds for $t = t_{n+1}$. To obtain optimal error estimates for both the velocity and the Lagrange multiplier, we carry out the following steps.
	
	\textbf{Step 1: Estimates of velocity. } By subtracting the continuous momentum equation \eqref{NS:orig:mom} evaluated at $t_{n+1}$ from the semi-discrete scheme \eqref{bdf1:mom:grad}, we obtain the following error equation for the velocity and pressure:
	\begin{subequations}\label{error:proj}
		\begin{empheq}[left=\empheqlbrace]{align}
			&\textstyle\frac{\hat{e}_{\bu}^{n+1}-e_{\bu}^n}{\tau}-\nu\Delta\hat{e}_{\bu}^{n+1}=-\mathcal{Q}(t_{n+1})\mathcal{N}(\bu(t_{n+1}))\bu(t_{n+1})+\mathcal{Q}^{n+1}\mathcal{N}(\bm u^n)\bm u^n\nonumber\\
			&\textstyle\qquad\qquad\qquad\qquad\qquad-\nabla(p(t_{n+1})-p^n)-\boldsymbol{R}_{\bu}^{n+1}, \label{err:mom:grad}\\ 
			&\textstyle\frac{e_{\bu}^{n+1}-\hat{e}_{\bu}^{n+1}}{\tau}-\nabla(p^{n+1}-p^n)=0,\label{err:proj:grad}
		\end{empheq} 
	\end{subequations}
	where the truncation error $\bR_{u}^{n+1}$ is given by
	\begin{equation}\label{trunc:velo}
		\bR_u^{n+1} = \bu_{t}(t_{n+1})-\frac{\bu(t_{n+1})-\bu(t_n)}{\tau}=\frac{1}{\tau}\int_{t_{n}}^{t_{n+1}} (t-t_n)u_{tt}(t) \, dt.
	\end{equation}
	Taking the inner product of \eqref{err:mom:grad} with the projected velocity error $\hat{e}_{\bu}^{n+1}$, we obtain
	\begin{equation}\label{err:mom:L2}
		\begin{aligned}
			&\textstyle\frac{\Vert\hat{e}_{\bu}^{n+1}\Vert^2-\Vert e_{\bu}^n\Vert^2}{2\tau}+\frac{\Vert\hat{e}_{\bu}^{n+1}-e_{\bu}^n\Vert^2}{2\tau}+\nu\Vert\nabla\hat{e}_{\bu}^{n+1}\Vert^2\\
			\textstyle=&\textstyle-\mathcal{Q}(t_{n+1})b(\bu(t_{n+1}),\bu(t_{n+1}),\hat{e}_{\bu}^{n+1})+\mathcal{Q}^{n+1}b(\bm u^n,\bm u^n,\hat{e}_{\bu}^{n+1})-(\nabla(p(t_{n+1})-p^n),\hat{e}_{\bu}^{n+1})\\
			&\textstyle-(\boldsymbol{R}_{\bu}^{n+1},\hat{e}_{\bu}^{n+1}),
		\end{aligned}
	\end{equation}
	Taking the inner product of \eqref{err:proj:grad} with $\frac{e_{\bu}^{n+1}-\hat{e}_{\bu}^{n+1}}{2}$, we derive
	\begin{equation}\label{err:proj:L2}
		\textstyle\frac{\Vert e_{\bu}^{n+1}\Vert^2-\Vert\hat{e}_{\bu}^{n+1}\Vert^2}{2\tau}-\frac{1}{2}(\nabla(p^{n+1}-p^n),\hat{e}_{\bu}^{n+1})=0
	\end{equation}
	Adding \eqref{err:mom:L2} and \eqref{err:proj:L2}, we have
	\begin{equation}\label{err:velo:add}
		\begin{aligned}
			&\textstyle\frac{\Vert e_{\bu}^{n+1}\Vert^2-\Vert e_{\bu}^n\Vert^2}{2\tau}+\frac{\Vert\hat{e}_{\bu}^{n+1}-e_{\bu}^n\Vert^2}{2\tau}+\nu\Vert\nabla\hat{e}_{\bu}^{n+1}\Vert^2\\
			\textstyle=&\textstyle-(\mathcal{Q}(t_{n+1})(\bu(t_{n+1})\cdot\nabla)\bu(t_{n+1})-\mathcal{Q}^{n+1}(\bm u^n\cdot\nabla)\bm u^n,\hat{e}_{\bu}^{n+1})-\frac{1}{2}(\nabla(2p(t_{n+1})-p^{n+1}-p^n),\hat{e}_{\bu}^{n+1})\\
			&\textstyle-(\boldsymbol{R}_{\bu}^{n+1},\hat{e}_{\bu}^{n+1}).
		\end{aligned}
	\end{equation}
	For the first term on the right-hand side of \eqref{err:velo:add}, we have
	\begin{equation*}
		\textstyle-(\mathcal{Q}(t_{n+1})(\bu(t_{n+1})\cdot\nabla)\bu(t_{n+1})-\mathcal{Q}^{n+1}(\bm u^n\cdot\nabla)\bm u^n,\hat{e}_{\bu}^{n+1})=\sum\limits_{i=1}^4B_i,
	\end{equation*}
	where 
	\begin{equation*}
		\begin{aligned}
			&\textstyle B_1=-b(\bu(t_{n+1})-\bu(t_n),\bu(t_{n+1}),\hat{e}_{\bu}^{n+1})-b(\bu(t_{n}),\bu(t_{n+1})-\bu(t_n),\hat{e}_{\bu}^{n+1}),\\
			&\textstyle B_2=-\mathcal{Q}^{n+1}b(e_{\bu}^n,\bu(t_n),\hat{e}_{\bu}^{n+1})-\mathcal{Q}^{n+1}b(\bu(t_n),e_{\bu}^n,\hat{e}_{\bu}^{n+1}),\\
			&\textstyle B_3=\mathcal{Q}^{n+1}b(e_{\bu}^n,e_{\bu}^n,\hat{e}_{\bu}^{n+1}),\\
			&\textstyle B_4=-e_{Q}^{n+1}b(\bu(t_n),\bu(t_n),\hat{e}_{\bu}^{n+1}),
		\end{aligned}
	\end{equation*}
	It follows from the properties of the trilinear form $b(\cdot,\cdot,\cdot)$ in \eqref{skew-C2} and the regularity assumption that
	\begin{equation}\label{err:B}
		\begin{aligned}
			\textstyle B_1\leq
			&\textstyle~2(1+C_1)C_2\Vert\bu(t_{n+1})-\bu(t_n)\Vert\Vert\nabla\hat{e}_{\bu}^{n+1}\Vert
			\leq~\frac{64(1+C_1)^2C_2^2}{\nu}\tau\int_{t_n}^{t_{n+1}}\Vert\bu_t\Vert^2dt+\frac{\nu}{64}\Vert\nabla\hat{e}_{\bu}^{n+1}\Vert^2,\\
			\textstyle B_2\leq
			&\textstyle~2(1+C_1) C_0 C_2M\Vert e_{\bu}^n\Vert\Vert\nabla\hat{e}_{\bu}^{n+1}\Vert
			\leq~\frac{64(1+C_1)^2 C_0^2C_2^2M^2}{\nu}\Vert e_{\bu}^n\Vert^2+\frac{\nu}{64}\Vert\nabla\hat{e}_{\bu}^{n+1}\Vert^2,\\
			\textstyle B_3\leq
			&\textstyle ~(1+C_1)^3 C_0 C_2\Vert\nabla e_{\bu}^n\Vert^2\Vert\nabla\hat{e}_{\bu}^{n+1}\Vert
			\leq~\frac{16(1+C_1)^6 C_0^2C_2^2}{\nu}\Vert\nabla e_{\bu}^n\Vert^4+\frac{\nu}{64}\Vert\nabla\hat{e}_{\bu}^{n+1}\Vert^2,\\
			\textstyle B_4\leq
			&\textstyle~C_2(1+C_1)M^2|e_Q^{n+1}|\Vert\nabla\hat{e}_{\bu}^{n+1}\Vert
			\leq~\frac{16}{\nu}C_2^2(1+C_1)^2M^4|e_Q^{n+1}|^2+\frac{\nu}{64}\Vert\nabla\hat{e}_{\bu}^{n+1}\Vert^2.
		\end{aligned}
	\end{equation}
	The third term on the right-hand side of \eqref{err:velo:add} can be bounded by
	\begin{equation}\label{err:p}
		\begin{aligned}
			&\textstyle-\frac{1}{2}(\nabla(2p(t_{n+1})-p^{n+1}-p^n),\hat{e}_{\bu}^{n+1})\\
			\textstyle=&\textstyle~\frac{\tau}{2}(\nabla(e_p^{n+1}+e_p^n+p(t_{n+1})-p(t_n)),\nabla(-e_p^{n+1}+e_p^n+p(t_{n+1})-p(t_n)))\\
			\textstyle=&\textstyle-\frac{\tau}{2}(\Vert\nabla e_p^{n+1}\Vert^2-\Vert\nabla e_p^n\Vert^2)+\tau(\nabla e_p^n,\nabla(p(t_{n+1})-p(t_n)))+\frac{\tau}{2}\Vert\nabla(p(t_{n+1})-p(t_n))\Vert^2\\
			\textstyle\leq&\textstyle-\frac{\tau}{2}(\Vert\nabla e_p^{n+1}\Vert^2-\Vert\nabla e_p^n\Vert^2)+\tau^2\Vert\nabla e_p^n\Vert^2+(1+\frac{\tau}{2})\tau\int_{t_n}^{t_{n+1}}\Vert\nabla p_t\Vert^2dt.
		\end{aligned}
	\end{equation}
	The inner product of the truncation error $R_{\bu}^{n+1}$ and $\hat{e}_{\bu}^{n+1}$ is estimated as follows:
	\begin{equation}\label{err:trunc}
		\begin{aligned}
			\textstyle -(\boldsymbol{R}_{\bu}^{n+1},\hat{e}_{\bu}^{n+1})\leq
			&\textstyle
			~\left\Vert\frac{1}{\tau}\int_{t_{n}}^{t_{n+1}}(t-t_{n})\bu_{tt}dt \right\Vert_{-1}\Vert e_{\bu}^{n+1}\Vert_{1}\leq(1+C_1)\left(\int_{t_{n}}^{t_{n+1}}\Vert\bu_{tt}\Vert_{-1}dt\right)\Vert\nabla \hat{e}_{\bu}^{n+1}\Vert\\
			\textstyle\leq
			&\textstyle~\frac{8(1+C_1)^2}{\nu}\tau\int_{t_{n}}^{t_{n+1}}\Vert\bu_{tt}\Vert_{-1}^2dt+\frac{\nu}{32}\Vert\nabla\hat{e}_{\bu}^{n+1}\Vert^2.
		\end{aligned}
	\end{equation}
	The combination of \eqref{err:velo:add}-\eqref{err:trunc} results in
	\begin{equation}\label{err:velo:sum}
		\begin{aligned}
			&\textstyle\frac{\Vert e_{\bu}^{n+1}\Vert^2-\Vert e_{\bu}^n\Vert^2}{2\tau}+\frac{\Vert\hat{e}_{\bu}^{n+1}-e_{\bu}^n\Vert^2}{2\tau}+\frac{\tau}{2}(\Vert\nabla e_p^{n+1}\Vert^2-\Vert\nabla e_p^n\Vert^2)+\frac{29\nu}{32}\Vert\nabla\hat{e}_{\bu}^{n+1}\Vert^2\\
			\textstyle\leq
			&\textstyle~\frac{64(1+C_1)^2C_2^2}{\nu}\tau\int_{t_n}^{t_{n+1}}\Vert\bu_t\Vert^2dt+\frac{64(1+C_1)^2 C_0^2C_2^2M^2}{\nu}\Vert e_{\bu}^n\Vert^2+\frac{16(1+C_1)^6 C_0^2C_2^2}{\nu}\Vert\nabla e_{\bu}^n\Vert^4\\
			&\textstyle+\frac{16}{\nu}C_2^2(1+C_1)^2M^4|e_Q^{n+1}|^2+\frac{8(1+C_1)^2}{\nu}\tau\int_{t_{n}}^{t_{n+1}}\Vert\bu_{tt}\Vert_{-1}^2dt\\
			&\textstyle+\tau^2\Vert\nabla e_p^n\Vert^2+(1+\frac{\tau}{2})\tau\int_{t_n}^{t_{n+1}}\Vert\nabla p_t\Vert^2dt.
		\end{aligned}
	\end{equation}
	Summing \eqref{err:velo:end} over $i=0,\cdots,n$ and multiplying the resulting estimate by $2\tau$, we deduce that
	\begin{equation}\label{err:velo:end}
		\begin{aligned}
			&\textstyle\Vert e_{\bu}^{n+1}\Vert^2+\sum\limits_{i=0}^n\Vert e_{\bu}^{i+1}-\hat{e}_{\bu}^{i+1}\Vert^2+\sum\limits_{i=0}^n\Vert\hat{e}_{\bu}^{i+1}-e_{\bu}^i\Vert^2+\tau^2\Vert\nabla e_p^{n+1}\Vert^2+\frac{29}{16}\nu\tau\sum\limits_{i=0}^n\Vert\nabla\hat{e}_{\bu}^{i+1}\Vert^2\\
			\textstyle\leq
			&\textstyle ~\frac{128(1+C_1)^2C_2^2}{\nu}\tau^2\int_{0}^{t_{n+1}}\Vert\bu_t\Vert^2dt+\frac{128(1+C_1)^2 C_0^2C_2^2M^2}{\nu}\tau\sum\limits_{i=0}^n\Vert e_{\bu}^i\Vert^2\\
			\textstyle
			&\textstyle~
			+\frac{32(1+C_1)^6 C_0^2C_2^2}{\nu}\tau\sum\limits_{i=0}^n\Vert\nabla e_{\bu}^i\Vert^4 +\frac{32}{\nu}C_2^2(1+C_1)^2M^4\tau\sum\limits_{i=0}^n|e_Q^{i+1}|^2\\
			\textstyle
			&\textstyle~
			+\frac{16(1+C_1)^2}{\nu}\tau^2\int_{0}^{t_{n+1}}\Vert\bu_{tt}\Vert_{-1}^2dt+2\tau^3\sum\limits_{i=0}^n\Vert\nabla e_p^i\Vert^2+(2+\tau)\tau^2\int_{0}^{t_{n+1}}\Vert\nabla p_t\Vert^2dt.
		\end{aligned}
	\end{equation}
	\textbf{Step 2: Estimates of the Lagrange multiplier.}  As the right-hand side of \eqref{err:velo:end} includes the error term $|e_Q^{n+1}|^2$ associated with the Lagrange multiplier, its estimation is necessary. This estimate is derived from the dynamic equation \eqref{bdf1:sav:grad}. Following the same approach as in \eqref{Q:err:orig:1}, an examination of \eqref{bdf1:sav:grad} reveals that it can be expressed in the following form:
	\begin{equation}\label{Q:err:orig}
		\begin{aligned}
			&\textstyle\theta\frac{(\mathcal{Q}^{n+1})^2-(\mathcal{Q})^2}{\tau}\\
			\textstyle=&\textstyle\frac{\Vert\bu^{n+1}-\hat{\bu}^{n+1}\Vert^2+\Vert\hat{\bu}^{n+1}-\bu^n\Vert^2}{2\tau}+\mathcal{Q}^{n+1}b(\bu^n,\bu^n,\hat{\bu}^{n+1})+(\nabla p^n,\hat{\bu}^{n+1})-\frac{\tau}{2}(\Vert\nabla p^{n+1}\Vert^2-\Vert\nabla p^n\Vert^2)\\
			\textstyle=&\textstyle\frac{\tau^2\Vert\nabla(p^{n+1}-p^n)\Vert^2+\Vert\hat{\bu}^{n+1}-\bu^n\Vert^2}{2\tau}+\mathcal{Q}^{n+1}b(\bu^n,\bu^n,\hat{\bu}^{n+1})+(\nabla p^n,\hat{\bu}^{n+1})-\frac{\tau}{2}(\Vert\nabla p^{n+1}\Vert^2-\Vert\nabla p^n\Vert^2)\\
			\textstyle=&\textstyle\frac{\Vert\hat{\bu}^{n+1}-\bu^n\Vert^2}{2\tau}+\mathcal{Q}^{n+1}b(\bu^n,\bu^n,\hat{\bu}^{n+1}).
		\end{aligned}
	\end{equation}
	This is accomplished by estimating $|\mathcal{Q}^{n+1}-1|$, taking into account that the exact value of $\mathcal{Q}$ is 1, $\mathcal{Q}^{n+1}>0$ and $\mathcal{Q}^n>0$. Then 
	\begin{equation}\label{Q:err:mid}
		\begin{aligned}
			\textstyle\theta\frac{e_Q^{n+1}-e_Q^n}{\tau}=&\textstyle\textstyle\frac{1}{\mathcal{Q}^{n+1}+\mathcal{Q}^n}\left(\frac{\Vert\hat{\bu}^{n+1}-\bu^n\Vert^2}{2\tau}+\mathcal{Q}^{n+1}b(\bu^n,\bu^n,\hat{\bu}^{n+1})\right).
			%\\
			%\textstyle\leq&\frac{1}{\mathcal{Q}^n}\frac{\Vert\hat{e}_{\bu}^{n+1}-%e_{\bu}^n\Vert^2+\Vert\bu(t_{n+1})-\bu(t_n)\Vert^2}%{\tau}+\frac{\mathcal{Q}^{n+1}}{\mathcal{Q}^n}b(\bu^n,\bu^n,\hat{\bu}^{n+1}).
		\end{aligned}
	\end{equation}
	Since we assume \eqref{u:Q:err} is valid for all $t=t_1, \cdots, t_n$ and $\tau\leq\frac{1}{(\frac{C_6C_8}{\nu}+C_9)^2}\leq\frac{1}{(C_6C_8+C_9)^2}$, we have $|e_Q^{n}|\leq (C_6C_8+C_9)\tau^2\leq\frac{\tau}{C_6C_8+C_9}$. By assume  $\tau\leq\frac{C_6C_8+C_9}{2}$, we have ${\frac12\leq\mathcal{Q}^n\leq\frac32}$. Taking $L^2$ inner product of \eqref{Q:err:mid} with $e_Q^{n+1}$, we obtain
	\begin{equation}\label{Q:err}
		\begin{aligned}
			\textstyle\theta\frac{|e_Q^{n+1}|^2-|e_Q^n|^2+|e_Q^{n+1}-e_Q^n|^2}{2\tau}\leq
			&\textstyle~2|e_Q^{n+1}|\left(\frac{\Vert\hat{e}_{\bu}^{n+1}-e_{\bu}^n\Vert^2+\Vert\bu(t_{n+1})-\bu(t_n)\Vert^2}{\tau}+\frac{\mathcal{Q}^{n+1}}{\mathcal{Q}^n}b(\bu^n,\bu^n,\hat{\bu}^{n+1})\right)\\
			\textstyle\leq
			&\textstyle~2|e_Q^{n+1}|\frac{\Vert\hat{e}_{\bu}^{n+1}-e_{\bu}^n\Vert^2+\Vert\bu(t_{n+1})-\bu(t_n)\Vert^2}{\tau}+2C_0|e_Q^{n+1}b(\bu^n,\bu^n,\hat{\bu}^{n+1})|\\
			\textstyle\leq
			&\textstyle~|e_Q^{n+1}|^2+2|e_Q^{n+1}|\frac{\Vert\hat{e}_{\bu}^{n+1}-e_{\bu}^n\Vert^2}{\tau}+\tau\int_{t_n}^{t_{n+1}}\Vert\bu_t\Vert^4dt+2C_0|e_Q^{n+1}b(\bu^n,\bu^n,\hat{\bu}^{n+1})|,
		\end{aligned}
	\end{equation}
	Using the properties of $b(\cdot,\cdot,\cdot)$ in \eqref{skew-C2}, we find
	\begin{equation}\label{Q:b}
		\begin{aligned}
			&\textstyle2C_0|e_Q^{n+1}b(\bu^n,\bu^n,\hat{\bu}^{n+1})|\\
			\textstyle=
			&\textstyle~2C_0|e_Q^{n+1}b(\bu(t_n)-e_{\bu}^n,\bu(t_n)-e_{\bu}^n,\bu(t_{n+1})-\hat{e}_{\bu}^{n+1})|\\
			\textstyle\leq
			&\textstyle~2C_0|e_Q^{n+1}\left(b(\bu(t_n),\bu(t_n),\bu(t_{n+1})-\bu(t_n))-b(\bu(t_n),\bu(t_n),\hat{e}_{\bu}^{n+1})-b(e_{\bu}^n,\bu(t_n),\bu(t_{n+1}))\right)|\\
			\textstyle
			&\textstyle~+2C_0|e_Q^{n+1}\left(b(\bu(t_n),e_{\bu}^n,\bu(t_{n+1}))-b(e_{\bu}^n,e_{\bu}^n,\bu(t_{n+1}))-b(e_{\bu}^n,\bu(t_n),\hat{e}_{\bu}^{n+1})\right)|\\
			&\textstyle+2C_0|e_Q^{n+1}\left(b(\bu(t_n),e_{\bu}^n,\hat{e}_{\bu}^{n+1})-b(e_{\bu}^n,e_{\bu}^n,\hat{e}_{\bu}^{n+1})\right)|\\
			\textstyle\leq
			&\textstyle~ C_0^2C_2^2\left(M^4(1+(1+C_1)^2\frac{192}{\theta\nu}+M^2(1+C_1)^4(1+\frac{128}{\theta\nu}\Vert\nabla\hat{e}_{\bu}^n\Vert^2))+\frac{64}{\theta\nu}\Vert\nabla\hat{e}_{\bu}^n\Vert^4\right)|e_Q^{n+1}|^2\\
			&\textstyle~+\tau\int_{t_n}^{t_{n+1}}\Vert\bu_t\Vert^2dt +\frac{\theta\nu}{16}\Vert\nabla\hat{e}_{\bu}^{n+1}\Vert^2+\frac{\theta\nu}{32}\Vert\nabla\hat{e}_{\bu}^n\Vert^2+\Vert\nabla\hat{e}_{\bu}^n\Vert^4.
		\end{aligned}
	\end{equation} 
	Since $\mathcal{Q}^n\leq C_0$ and $\mathcal{Q}(t)=0$, it follows that $|e_Q^n|\leq1+C_0$.  Under this bound, we may assume $\theta\geq 8(1+C_0)$. Based on the previous assumption, we can conclude that $\Vert\nabla\hat{e}_u^n\Vert^2\leq C_n$.
	
	Combining \eqref{Q:b}  with \eqref{Q:err}, we have
	\begin{equation}\label{Q:error}
		\begin{aligned}
			\textstyle\frac{|e_Q^{n+1}|^2-|e_Q^n|^2+|e_Q^{n+1}-e_Q^n|^2}{2\tau}\leq&\textstyle C_3|e_Q^{n+1}|^2+\frac14\frac{\Vert\hat{e}_{\bu}^{n+1}-e_{\bu}^n\Vert^2}{\tau}+\frac{2\tau}{\theta}\int_{t_n}^{t_{n+1}}\Vert\bu_t\Vert^4dt\\
			&\textstyle+\frac{\tau}{\theta}\int_{t_n}^{t_{n+1}}\Vert\bu_t\Vert^2dt +\frac{\nu}{16}\Vert\nabla\hat{e}_{\bu}^{n+1}\Vert^2+\frac{\nu}{32}\Vert\nabla\hat{e}_{\bu}^n\Vert^2+\frac{1}{\theta}\Vert\nabla\hat{e}_{\bu}^n\Vert^4,
		\end{aligned}
	\end{equation}
	where 
	\begin{equation*}
		\textstyle C_3=\frac{C_0^2C_2^2}{\theta}\left(M^4(1+(1+C_1)^2\frac{192}{\theta\nu}+M^2(1+C_1)^4(1+\frac{128}{\theta\nu}C_n))+\frac{64}{\theta\nu}C_n^2\right)+\frac{1}{\theta}.
	\end{equation*}
	Summing \eqref{Q:error} over $n=0,\cdots,k$ and multiplying $2\tau$ on the both sides, we obtain
	\begin{equation}\label{Q:error:sum}
		\begin{aligned}
			\textstyle|e_Q^{k+1}|^2+\sum\limits_{i=0}^k|e_Q^{i+1}-e_Q^i|^2\leq
			&\textstyle2\sum\limits_{i=0}^kC_3\tau|e_Q^{i+1}|^2+\sum\limits_{i=0}^k\frac{\Vert\hat{e}_{\bu}^{i+1}-e_{\bu}^i\Vert^2}{4\tau}+\frac{4\tau^2}{\theta}\int_{0}^{t_{k+1}}\Vert\bu_t\Vert^4dt\\
			&\textstyle+\frac{2\tau^2}{\theta}\int_{0}^{t_{k+1}}\Vert\bu_t\Vert^2dt +\frac{3\nu}{16}\tau\sum\limits_{i=0}^k\Vert\nabla\hat{e}_{\bu}^{i+1}\Vert^2+\frac{2\tau}{\theta}\sum\limits_{i=0}^k\Vert\nabla\hat{e}_{\bu}^i\Vert^4.
		\end{aligned}
	\end{equation}
	\textbf{Step 3: Simultaneous estimates for the velocity and Lagrange multiplier.} Combining Step 1 and Step 2 (let $k=n$ in \eqref{Q:error:sum}),  we obtain
	\begin{equation}\label{velo:lagrange:err}
		\begin{aligned}
			&\textstyle\Vert e_{\bu}^{n+1}\Vert^2+\frac{1}{2}\sum\limits_{i=0}^n\Vert\hat{e}_{\bu}^{i+1}-e_{\bu}^i\Vert^2+|e_Q^{n+1}|^2+\sum\limits_{i=0}^n|e_Q^{i+1}-e_Q^i|^2+\tau^2\Vert\nabla e_p^{n+1}\Vert^2+\frac{13}{8}\nu\tau\sum\limits_{i=0}^n\Vert\nabla\hat{e}_{\bu}^{i+1}\Vert^2\\
			\textstyle\leq
			&\textstyle~\frac{128(1+C_1)^2C_2^2}{\nu}\tau^2\int_{0}^{t_{n+1}}\Vert\bu_t\Vert^2dt+\frac{128(1+C_1)^2 C_0^2C_2^2M^2}{\nu}\tau\sum\limits_{i=0}^n\Vert e_{\bu}^i\Vert^2\\
			\textstyle
			&\textstyle~+\sum\limits_{i=0}^n(\frac{32}{\nu}C_2^2(1+C_1)^2M^4+2C_3)\tau|e_Q^{i+1}|^2+\frac{16(1+C_1)^2}{\nu}\tau^2\int_{0}^{t_{n+1}}\Vert\bu_{tt}\Vert_{-1}^2dt\\
			&\textstyle~+2\tau^3\sum\limits_{i=0}^n\Vert\nabla e_p^i\Vert^2+(2+\tau)\tau^2\int_{0}^{t_{n+1}}\Vert\nabla p_t\Vert^2dt\\
			\textstyle
			&\textstyle~+\frac{2\tau^2}{\theta}\int_{0}^{t_{n+1}}(2\Vert\bu_t\Vert^4+\Vert\bu_t\Vert^2)dt +\left(\frac{32(1+C_1)^6 C_0^2C_2^2}{\nu}+\frac{2}{\theta}\right)\tau\sum\limits_{i=0}^n\Vert\nabla\hat{e}_{\bu}^i\Vert^4.
		\end{aligned}
	\end{equation}
	Applying the discrete Gr\"onwall's inequality, with $a_n=\Vert e_{\bu}^{n+1}\Vert^2+|e_Q^{n+1}|^2+\tau^2\Vert\nabla e_p^{n+1}\Vert^2$ for $0\leq n\leq N$, we derive
	\begin{equation}\label{err:Gronwall}
		\begin{aligned}
			&\textstyle\Vert e_{\bu}^{n+1}\Vert^2+\frac12\sum\limits_{i=0}^n\Vert\hat{e}_{\bu}^{i+1}-e_{\bu}^i\Vert^2+|e_Q^{n+1}|^2+\tau^2\Vert\nabla e_p^{n+1}\Vert^2+\frac{13}{8}\nu\tau\sum\limits_{i=0}^n\Vert\nabla\hat{e}_{\bu}^{i+1}\Vert^2\\
			\textstyle
			\leq&\textstyle
			~C_6 \left(\frac{128(1+C_1)^2C_2^2}{\nu}\tau^2\int_{0}^{T}\Vert\bu_t\Vert^2dt+\frac{16(1+C_1)^2}{\nu}\tau^2\int_{0}^{T}\Vert\bu_{tt}\Vert_{-1}^2dt+ C_7\tau\sum\limits_{i=0}^n\Vert\nabla\hat{e}_{\bu}^i\Vert^4
			\right.\\
			&\textstyle\left.~+(2+\tau)\tau^2\int_{0}^{T}\Vert\nabla p_t\Vert^2dt+\frac{4}{\theta\tau}\sum\limits_{i=0}^n\Vert\hat{e}_{\bu}^{i+1}-e_{\bu}^i\Vert^4+\frac{2\tau^2}{\theta}\int_{0}^{T}(2\Vert\bu_t\Vert^4+\Vert\bu_t\Vert^2)dt\right)\\
			\textstyle
			\textstyle\leq&\textstyle~C_6\left( (\frac{128(1+C_1)^2C_2^2M}{\nu}+\frac{16(1+C_1)^2M}{\nu}+(2+\tau)\int_{0}^{T}\Vert\nabla p_t\Vert^2dt+\frac{6M}{\theta})\tau^2+ C_7\tau\sum\limits_{i=0}^n\Vert\nabla\hat{e}_{\bu}^i\Vert^4\right)\\
			\textstyle\leq&\textstyle~C_6C_8\tau^2+ C_6C_7\tau\sum\limits_{i=0}^n\Vert\nabla\hat{e}_{\bu}^i\Vert^4,
		\end{aligned}
	\end{equation}
	where
	\begin{equation*}
		\begin{aligned}
			&\textstyle C_6=\exp\left(\tau\sum\limits_{i=0}^n(\frac{128(1+C_1)^2 C_0^2C_2^2M^2}{\nu}+\frac{32}{\nu}C_2^2(1+C_1)^2M^4+2C_3+2)\right),\\
			&\textstyle C_7=\frac{32(1+C_1)^6 C_0^2C_2^2}{\nu}+\frac{2}{\theta},\\
			&\textstyle C_8=\frac{128(1+C_1)^2C_2^2M}{\nu} +\frac{16(1+C_1)^2M}{\nu}+(2+\tau)\int_{0}^{T}\Vert\nabla p_t\Vert^2dt+\frac{6M}{\theta}.
		\end{aligned}
	\end{equation*}
	\textbf{Step 4:} Estimates for $\sum\limits_{i=0}^n\Vert\nabla\hat{e}_{\bu}^n\Vert^4$ by induction. Let $C_9=\frac{C_6C_7}{\nu}$ it follows from \eqref{err:Gronwall} and $\text{poincar\'{e}}$ inequality that\\
	\begin{equation*}
		\textstyle\nu\tau\sum\limits_{i=0}^n\Vert\nabla\hat{e}_{\bu}^{i+1}\Vert^2\leq\frac{13}{8}\nu\tau\sum\limits_{i=0}^n\Vert\nabla\hat{e}_{\bu}^{i+1}\Vert^2\leq C_6C_8\tau^2+ C_6C_7\tau(\sum\limits_{i=0}^n\Vert\nabla\hat{e}_{\bu}^i\Vert^2)^2, 
	\end{equation*}
	which leads to
	\begin{equation}\label{nabla:e^4}
		\textstyle\sum\limits_{i=0}^{n+1}\Vert\nabla\hat{e}_{\bu}^{i}\Vert^2\leq [\frac{C_6C_8}{\nu}+ \frac{C_9}{\tau}(\sum\limits_{i=0}^{n}\Vert\nabla\hat{e}_{\bu}^i\Vert^2)^2]\tau,
	\end{equation}
	Using \eqref{nabla:e^4}, we will prove by induction that, for sufficiently small $\tau$, the following holds for $0\leq n\leq N$:
	\begin{equation}\label{nabla:e^4:alpha}
		\textstyle\sum\limits_{i=0}^{n+1}\Vert\nabla\hat{e}_{\bu}^i\Vert^2\leq\left(\frac{C_6C_8}{\nu}+ C_9\right)\tau:=\alpha\tau,
	\end{equation}
	Since $ \hat e_{\bu}^0=e_{\bu}^0$ and $e_{\bu}^0=0$, we can obtain $\textstyle\Vert\nabla\hat{e}_{\bu}^1\Vert^2\leq\frac{C_6C_8}{\nu}\tau\leq\left(\frac{C_6C_8}{\nu}+ C_9\right)\tau$ from \eqref{nabla:e^4}. Clearly, \eqref{nabla:e^4:alpha} is verified for $n=0$. Suppose it is true for some nonnegative integer $n < N$. In light of \eqref{nabla:e^4}, for $\tau$ satisfying
	\begin{equation*}\label{tau:bound1}
		\tau\leq\frac{1}{\alpha^2}=\frac{1}{(\frac{C_6C_8}{\nu}+ C_9)^2},
	\end{equation*}
	we find that
	\begin{equation*}
		\textstyle\sum\limits_{i=0}^{n+1}\Vert\nabla\hat{e}_{\bu}^i\Vert^2\leq(\frac{C_6C_8}{\nu}+ C_9\alpha^2\tau)\tau\leq(\frac{C_6C_8}{\nu}+ C_9)\tau=\alpha\tau,
	\end{equation*}
	which verifies the statement \eqref{nabla:e^4:alpha}, provided that $\tau\leq\frac{1}{\alpha^2}$. Therefore, we have the following estimate
	\begin{equation}\label{e^4:bound}
		\textstyle C_6C_7\tau\sum\limits_{i=0}^n\Vert\nabla\hat{e}_{\bu}^i\Vert^4\leq C_9\tau(\sum\limits_{i=0}^{n}\Vert\nabla\hat{e}_{\bu}^i\Vert^2)^2\leq C_9\alpha^2\tau^3\leq C_9\tau^2.
	\end{equation}
	\textbf{Step 5: Optimal convergence rates for the velocity and Lagrange multiplier.}  Combining \eqref{err:Gronwall} and \eqref{e^4:bound}, we end up with
	\begin{equation*}
		\textstyle\Vert e_{\bu}^{n+1}\Vert^2+\sum\limits_{i=0}^n\Vert\hat{e}_{\bu}^{i+1}-e_{\bu}^i\Vert^2+|e_Q^{n+1}|^2+\tau^2\Vert\nabla e_p^{n+1}\Vert^2+\frac{13}{8}\nu\tau\sum\limits_{i=0}^n\Vert\nabla\hat{e}_{\bu}^{i+1}\Vert^2\leq (C_6C_8+C_9)\tau^2.
	\end{equation*}
	Consequently, we arrive at \eqref{u:Q:err}.
\end{proof}
{\color{blue}
	\begin{remark}
		Since $\hat{\bu}^{n+1}$ is not divergence-free, we have one extra term compared to the analysis in \cite{Doan2025IMA}, that is $ -\frac{1}{2}(\nabla(2p(t_{n+1}) - p^{n+1} - p^n), \hat{e}_{\boldsymbol{u}}^{n+1}) $ in Eq. \eqref{err:velo:add}. It can be further estimated by Eq. \eqref{err:p}.
	\end{remark}
}
\subsection{Optimal error estimates for the pressure}
The main result in this subsection is the following optimal $l^2(0,T;L^2(\Omega))$ error estimate for the pressure.% which requires additional regularities.
\begin{thm}\label{pressure estimate}
	Assuming $\bu\in H^3(0,T;L^2(\Omega))H^3(0,T;L^2(\Omega))\bigcap H^1(0,T;H_0^2(\Omega))\bigcap W^{1, \infty }( 0, T; W^{1, \infty }( \Omega ) )$, $p\in H^2( 0, T; H^1( \Omega ) )$, then the first-order scheme \eqref{discretization:1or} hold for all $\tau\leq \max\{\frac{3\nu^2}{256C_0^2C_2^2(1+C_1)^6},\tau_0\}$ and $\theta\geq\max\{8(1+C_0),4C_2(1+C_1)\sqrt{\frac{C_5}{\nu}(M^4+(1+C_1)^4C)}\}$:
	\begin{equation}
		\textstyle\tau\sum\limits_{i=0}^n\Vert e_p^{i}\Vert_{L^2(\Omega)/R}^2\leq C\tau^2,\quad\forall\:0\leq n\leq N,
	\end{equation}
	where $C$ is a positive constant independent of $\tau$.
\end{thm}
\begin{proof}
	Our strategy for obtaining optimal pressure error estimates is to first use the inf-sup condition to relate the pressure error to the velocity error and the Lagrange multiplier. The key step is the estimation of the discrete time derivative of the velocity error, which requires suitable regularity assumptions on the exact solution $\bu$. The desired estimate is then concluded by applying the discrete Gr\"onwall inequality (Lemma \ref{discrete Gronwall}).
	
	\textbf{Step 1: Estimates of the pressure.}  Taking notice of the fact that
	\begin{equation}\label{infsup}
		\textstyle\Vert e_p^n\Vert\leq \sup\limits_{ \bv \in H_0^1(\Omega)}\frac{(\nabla e_p^n,\bv)}{\Vert\nabla \bv\Vert}.
	\end{equation}
	Therefore, in order to bound $\Vert e_p^n\Vert $ for $1\leq n\leq N$, it is sufficient to estimate the negative norm $\Vert \nabla e_p^n\Vert _{-1}$. From \eqref{error:proj}, we have
	\begin{equation}\label{p:-1}
		\textstyle\nabla e_{p}^{n+1}=-\frac{e_{\bu}^{n+1}-e_{\bu}^{n}}{\tau}+\nu\Delta \hat{e}_{\bu}^{n+1}-\boldsymbol{R}_{\bu}^{n+1}-(\bu(t_{n+1})\cdot\nabla)\bu(t_{n+1})+\mathcal{Q}^{n+1}(\bu^{n}\cdot\nabla)\bu^{n}.
	\end{equation}
	For any $\bv \in H_0^1(\Omega)$, the $L^2$ inner product of \eqref{p:-1} with $\bv$ leads to
	\begin{equation}\label{p:v}
		\begin{aligned}
			\textstyle(\nabla e_p^{n+1}, \bv) =&\textstyle -\frac{1}{\tau}( e_{\bu}^{n+1}-e_{\bu}^n, \bv)-\nu( \nabla \hat{e}_{\bu}^{n+1}, \nabla \bv)-(\boldsymbol{R}_{\bu}, \bv) \\
			&\textstyle-b( \bu(t_{n+1}), \bu(t_{n+1}), \bv) + \mathcal{Q}^{n+1} b( \bu^n, \bu^n, \bv) \\
			\textstyle
			\leq&~\textstyle \frac{1+C_1}{\tau}\Vert e_{\bu}^{n+1} - e_{\bu}^n\Vert_{-1}\Vert\nabla \bv\Vert + \nu\Vert\nabla \hat{e}_{\bu}^{n+1}\Vert\Vert\nabla \bv\Vert +(1+C_1)\int_{t_n}^{t_{n+1}} \Vert\bu_{tt}\Vert_{-1} dt\Vert\nabla \bv\Vert \\
			&\textstyle + 2C_0C_2(1+C_1) M \int_{t_n}^{t_{n+1}}\Vert\bu_t\Vert dt \Vert\nabla \bv\Vert + 2C_0 C_2(1+C_1) M \Vert e_{\bu}^n\Vert\Vert\nabla \bv\Vert  \\
			&\textstyle+ C_0 C_2(1+C_1)^3\Vert\nabla e_{\bu}^n\Vert^2\Vert\nabla \bv\Vert + C_2 (1+C_1)M^2 |e_Q^{n+1}|\Vert\nabla \bv\Vert.
		\end{aligned}
	\end{equation}
	Combining \eqref{infsup} and \eqref{p:v}, we obtain the following:
	\begin{equation}\label{p:inf-sup}
		\begin{aligned}
			\textstyle\Vert e_p^{n+1}\Vert\leq&~\textstyle\frac{1+C_1}{\tau}\Vert e_{\bu}^{n+1} - e_{\bu}^n\Vert_{-1}+ \nu\Vert\nabla \hat{e}_{\bu}^{n+1}\Vert+ (1+C_1)\int_{t_n}^{t_{n+1}} \Vert\bu_{tt}\Vert_{-1} dt+ 2C_0C_2(1+C_1) M \int_{t_n}^{t_{n+1}}\Vert\bu_t\Vert dt \\
			&\textstyle + 2C_0 C_2(1+C_1) M \Vert e_{\bu}^n\Vert+C_0 C_2(1+C_1)^3\Vert\nabla e_{\bu}^n\Vert^2+ C_2 (1+C_1)M^2 |e_Q^{n+1}|.
		\end{aligned}
	\end{equation}
	Squaring both sides of \eqref{p:inf-sup} and applying the Cauchy-Schwarz inequality, we deduce that
	\begin{equation}\label{p:Cauchy}
		\begin{aligned}
			\textstyle\Vert e_p^{n+1}\Vert^2\leq&~((1+C_1)^2+\nu+(1+C_1)^2+8C_0^2C_2^2(1+C_1)^2+C_0^2C_2^2(1+C_1)^6+ C_2^2(1+C_1)^2M^4)\\
			&\textstyle\times(\frac{1}{\tau^2}\Vert e_{\bu}^{n+1} - e_{\bu}^n\Vert_{-1}^2+ \nu\Vert\nabla \hat{e}_{\bu}^{n+1}\Vert^2+ \tau\int_{t_n}^{t_{n+1}} \Vert\bu_{tt}\Vert_{-1}^2 dt+\tau\int_{t_n}^{t_{n+1}}\Vert\bu_t\Vert^2 dt \\
			&\textstyle\qquad\qquad +\Vert e_{\bu}^n\Vert^2+\Vert\nabla e_{\bu}^n\Vert^4+|e_Q^{n+1}|^2).
		\end{aligned}
	\end{equation}
	Let $\hat{C}=2(1+C_1)^2+\nu+8C_0^2C_2^2(1+C_1)^2+C_0^2C_2^2(1+C_1)^6+ C_2^2(1+C_1)^2M^4$. Changing the index $n$ to $i$ in \eqref{p:Cauchy}, and summing over  from $i = 0$ to $n$  yields
	\begin{equation}\label{p:sum}
		\begin{aligned}
			\textstyle\sum\limits_{i=0}^n\Vert e_p^{i+1}\Vert^2\leq&\textstyle\hat{C}(\frac{1}{\tau^2}\sum\limits_{i=0}^n\Vert e_{\bu}^{i+1} - e_{\bu}^i\Vert_{-1}^2+\nu\sum\limits_{i=0}^n\Vert\nabla \hat{e}_{\bu}^{i+1}\Vert^2+\tau\int_{0}^{t_{n+1}} \Vert\bu_{tt}\Vert_{-1}^2 dt+\tau\int_{0}^{t_{n+1}}\Vert\bu_t\Vert^2 dt\\
			&\textstyle+\sum\limits_{i=0}^n\Vert e_{\bu}^i\Vert^2+\sum\limits_{i=0}^n\Vert\nabla e_{\bu}^i\Vert^4+\sum\limits_{i=0}^n|e_Q^{i+1}|^2).
		\end{aligned}
	\end{equation}
	\textbf{Step 2: Establish an estimate on $\Vert e_{\bu}^{n+1}-e_{\bu}^n\Vert_{-1}$.} Since $\Vert e_{\bu}^{n+1}-e_{\bu}^n\Vert_{-1}\leq\Vert e_{\bu}^{n+1}-e_{\bu}^n\Vert$, we will bound $\Vert e_{\bu}^{n+1}-e_{\bu}^n\Vert$ instead.
	The difference of \eqref{error:proj} between two consecutive time steps is
	\begin{equation}\label{diff:two time:hate}
		\begin{aligned}
			\textstyle\frac{d_{t}\hat{e}_{\bu}^{n+1}-d_{t}e_{\bu}^{n}}{\tau}-\nu\Delta d_{t}\hat{e}_{\bu}^{n+1}+\nabla d_t(p(t_{n+1})-p^n) =&\textstyle-d_{t}\boldsymbol{R}_{\bu}^{n+1}-\frac{(\bu(t_{n+1})\cdot\nabla)\bu(t_{n+1})-(\bu(t_{n})\cdot\nabla)\bu(t_{n})}{\tau} \\
			&\textstyle +\frac{\mathcal{Q}^{n+1}(\bu^{n}\cdot\nabla)\bu^{n}-\mathcal{Q}^{n}(\bu^{n-1}\cdot\nabla)\bu^{n-1}}{\tau}.
		\end{aligned}
	\end{equation}
	where $d_t\phi^n:=\frac{\phi^n-\phi^{n-1}}{\tau}$ is the backward time difference, and
	\begin{equation}\label{diff:two time:e}
		\begin{aligned}
			\textstyle \frac{d_{t}e_{\bu}^{n+1}-d_{t}\hat{e}_{\bu}^{n+1}}{\tau}-\nabla d_t(p^{n+1}-p^n)=0.
		\end{aligned}
	\end{equation}
	Taking the $L^2$ inner product of \eqref{diff:two time:hate} with $d_t\hat{e}_{\bu}^{n+1}$ and \eqref{diff:two time:e} with $\frac{d_t e_{\bu}^{n+1}+d_t\hat{e}_{\bu}^{n+1}}{2}$ gives us
	\begin{equation}\label{diff:hate:L2}
		\begin{aligned}
			&\textstyle\frac{\Vert d_{t}\hat{e}_{\bu}^{n+1}\Vert^2-\Vert d_{t}e_{\bu}^{n}\Vert^2+\Vert d_{t}\hat{e}_{\bu}^{n+1}-d_{t}e_{\bu}^{n}\Vert^2}{2\tau}+\nu\Vert\nabla d_{t}\hat{e}_{\bu}^{n+1}\Vert^2+(\nabla d_t(p(t_{n+1})-p^n),d_{t}\hat{e}_{\bu}^{n+1})\\
			\textstyle=&\textstyle-(d_{t}\boldsymbol{R}_{\bu}^{n+1},d_{t}\hat{e}_{\bu}^{n+1})-\frac{b(\bu(t_{n+1}),\bu(t_{n+1}),d_{t}\hat{e}_{\bu}^{n+1})-b(\bu(t_{n}),\bu(t_{n}),d_{t}\hat{e}_{\bu}^{n+1})}{\tau} \\
			&\textstyle +\frac{\mathcal{Q}^{n+1}b(\bu^{n},\bu^{n},d_{t}\hat{e}_{\bu}^{n+1})-\mathcal{Q}^{n}b(\bu^{n-1},\bu^{n-1},d_{t}\hat{e}_{\bu}^{n+1})}{\tau},
		\end{aligned}
	\end{equation}
	and
	\begin{equation}\label{diff:e:L2}
		\textstyle\frac{\Vert d_{t}e_{\bu}^{n+1}\Vert^2-\Vert d_{t}\hat{e}_{\bu}^{n+1}\Vert^2}{2\tau}-\frac{1}{2}(\nabla d_t(p^{n+1}-p^n), d_{t}\hat{e}_{\bu}^{n+1})=0.
	\end{equation}
	Summing \eqref{diff:hate:L2} and \eqref{diff:e:L2}, we derive
	\begin{equation}\label{dteu}
		\begin{aligned}
			&\textstyle\frac{\Vert d_{t}e_{\bu}^{n+1}\Vert^2-\Vert d_{t}e_{\bu}^n\Vert^2}{2\tau}+\frac{\Vert d_{t}\hat{e}_{\bu}^{n+1}-d_{t}e_{\bu}^{n}\Vert^2}{2\tau}+\nu\Vert\nabla d_{t}\hat{e}_{\bu}^{n+1}\Vert^2\\
			\textstyle=&\textstyle-(d_{t}\boldsymbol{R}_{\bu}^{n+1},d_{t}\hat{e}_{\bu}^{n+1})-\frac12(\nabla d_t(2p(t_{n+1})-p^n-p^{n+1}),d_t\hat{e}_{\bu}^{n+1})\\
			&\textstyle-\frac{b(\bu(t_{n+1}),\bu(t_{n+1}),d_{t}\hat{e}_{\bu}^{n+1})-b(\bu(t_{n}),\bu(t_{n}),d_{t}\hat{e}_{\bu}^{n+1})}{\tau} \\
			&\textstyle+\frac{\mathcal{Q}^{n+1}b(\bu^{n},\bu^{n},d_{t}\hat{e}_{\bu}^{n+1})-\mathcal{Q}^{n}b(\bu^{n-1},\bu^{n-1},d_{t}\hat{e}_{\bu}^{n+1})}{\tau}\\
			\textstyle=&\textstyle-(d_{t}\boldsymbol{R}_{\bu}^{n+1},d_{t}\hat{e}_{\bu}^{n+1})
			-\frac12(\nabla d_t(2p(t_{n+1})-p^n-p^{n+1}),d_t\hat{e}_{\bu}^{n+1})+S_1+S_2,
		\end{aligned}
	\end{equation}
	where $S_1$ and $S_2$ are defined as
	\begin{equation*}
		\begin{aligned}
			\textstyle S_{1}  =&\textstyle-\frac{b(\bu(t_{n+1}),\bu(t_{n+1}),d_t\hat{e}_{\bu}^{n+1})-2b(\bu(t_n),\bu(t_n),d_t\hat{e}_{\bu}^{n+1})+b(\bu(t_{n-1}),\bu(t_{n-1}),d_t\hat{e}_{\bu}^{n+1})}{\tau}, \\
			\textstyle S_{2}  =&\textstyle\frac{\mathcal{Q}^{n+1}b(\bu^n,\bu^n,d_t\hat{e}_{\bu}^{n+1})-b(\bu(t_n),\bu(t_n),d_t\hat{e}_{\bu}^{n+1})}{\tau} \\
			&\qquad\qquad\textstyle -\frac{\mathcal{Q}^nb(\bu^{n-1},\bu^{n-1},d_t\hat{e}_{\bu}^{n+1})-b(\bu(t_{n-1}),\bu(t_{n-1}),d_t\hat{e}_{\bu}^{n+1})}{\tau}.
		\end{aligned}
	\end{equation*}
	First of all, the term  $S_1$ can be reformulated as follows:
	\begin{equation*}
		\begin{aligned}
			\textstyle S_{1}= & \textstyle-\frac{b(\bu(t_{n+1})-2\bu(t_n)+\bu(t_{n-1}),\bu(t_{n-1}),d_t\hat{e}_{\bu}^{n+1})}{\tau}-\frac{b(\bu(t_n),\bu(t_{n+1})-2\bu(t_n)+\bu(t_{n-1}),d_t\hat{e}_{\bu}^{n+1})}{\tau} \\
			&\textstyle -\frac{b(\bu(t_{n+1})-\bu(t_n),\bu(t_{n+1})-\bu(t_{n-1}),d_t\hat{e}_{\bu}^{n+1})}{\tau}.
		\end{aligned}
	\end{equation*}
	Using the properties of $b(\cdot,\cdot,\cdot)$ in \eqref{skew-C2} and the regularity assumption, we find
	\begin{equation}\label{S_1}
		\begin{aligned}
			\textstyle S_{1}
			\leq
			&~\textstyle\frac{C_2}{\tau}\Vert \bu(t_{n+1})-2\bu(t_n)+\bu(t_{n-1})\Vert _0\Vert \bu(t_{n-1})\Vert _2\Vert d_t\hat{e}_{\bu}^{n+1}\Vert _1 \\
			&\textstyle+\frac{C_2}{\tau}\Vert \bu(t_n)\Vert _2\Vert \bu(t_{n+1})-2\bu(t_n)+\bu(t_{n-1})\Vert _0\Vert d_t\hat{e}_{\bu}^{n+1}\Vert _1 \\
			&\textstyle+\frac{C_2}{\tau}\Vert \bu(t_{n+1})-\bu(t_n)\Vert _1\Vert \bu(t_{n+1})-\bu(t_{n-1})\Vert _1\Vert d_t\hat{e}_{\bu}^{n+1}\Vert _1 \\
			\textstyle\leq&~\textstyle(1+C_1)\left(2C_2M\int_{t_{n-1}}^{t_{n+1}}\Vert \bu_{tt}\Vert _0dt+2C_2\int_{t_{n-1}}^{t_{n+1}}\Vert \bu_t\Vert _1^2dt\right)\Vert \nabla d_t\hat{e}_{\bu}^{n+1}\Vert .
		\end{aligned}
	\end{equation}
	Secondly, $S_2$ can be expressed as
	\begin{equation*}
		\begin{aligned}
			\textstyle S_2 =&~\textstyle d_t e_Q^{n+1} b(\bu^n, \bu^n, d_t\hat{e}_{\bu}^{n+1})+e_Q^{n}\frac{b(\bu(t_n), \bu(t_n) - \bu(t_{n-1}), d_t \hat{e}_{\bu}^{n+1}) + b(\bu(t_n) - \bu(t_{n-1}), \bu(t_{n-1}), d_t\hat{e}_{\bu}^{n+1})}{\tau}\\
			&\textstyle+\mathcal{Q}^n \frac{b(e^{n-1}_{\bu}, \bu(t_n) - \bu(t_{n-1}), d_t\hat{e}^{n+1}_{\bu}) + b(\bu(t_n) - \bu(t_{n-1}), e^{n}_{\bu}, d_t\hat{e}^{n+1}_{\bu})}{\tau}\\
			&\textstyle+\mathcal{Q}^n \left[ b(d_t e^{n}_{\bu}, \bu(t_n), d_t\hat{e}^{n+1}_{\bu}) + b(\bu(t_{n-1}), d_t e^{n}_{\bu}, d_t\hat{e}^{n+1}_{\bu}) \right]\\
			&\textstyle+\mathcal{Q}^n \left[ b(d_t e^{n}_{\bu}, e^{n-1}_{\bu}, d_t\hat{e}^{n+1}_{\bu}) + b(e^{n}_{\bu}, d_t e^{n}_{\bu}, d_t\hat{e}^{n+1}_{\bu}) \right].
		\end{aligned}
	\end{equation*}
	Similarly, applying the properties of the trilinear form $b(\cdot,\cdot,\cdot)$ specified in \eqref{skew-C2}, we obtain:
	\begin{equation}\label{S_2}
		\begin{aligned}
			\textstyle S_2  \leq&~\textstyle C_2|d_te_Q^{n+1}|\Vert \bu^n\Vert _1^2\Vert d_t\hat{e}_{\bu}^{n+1}\Vert _1+\frac{2C_2M}\tau|e_Q^n|\left(\int_{t_{n-1}}^{t_n}\Vert \bu_t\Vert _0dt\right)\Vert d_t\hat{e}_{\bu}^{n+1}\Vert _1 \\
			&\textstyle+\frac{C_0C_2}\tau\left(\Vert e_{\bu}^{n-1}\Vert _0+\Vert e_{\bu}^n\Vert _0\right)\left(\int_{t_{n-1}}^{t_n}\Vert \bu_t\Vert _2dt\right)\Vert d_t\hat{e}_{\bu}^{n+1}\Vert _1 \\
			&\textstyle +2C_0C_2M\Vert d_te_{\bu}^n\Vert _0\Vert d_t\hat{e}_{\bu}^{n+1}\Vert _1+C_0C_2\left(\Vert e_{\bu}^{n-1}\Vert _1+\Vert e_{\bu}^n\Vert _1\right)\Vert d_te_{\bu}^n\Vert _1\Vert d_te_{\bu}^{n+1}\Vert _1\\
			\textstyle\leq&~\textstyle C_2(1+C_1)|d_te_Q^{n+1}|\left(\Vert \bu(t_n)\Vert _1^2+(1+C_1)^2\Vert\nabla\hat{e}_{\bu}^n\Vert^2\right)\Vert\nabla d_t\hat{e}_{\bu}^{n+1}\Vert\\
			&\textstyle+\frac{2C_2(1+C_1)M}\tau|e_Q^n|\left(\int_{t_{n-1}}^{t_n}\Vert \bu_t\Vert _0dt\right)\Vert\nabla d_t\hat{e}_{\bu}^{n+1}\Vert +2C_0C_2(1+C_1)M\Vert d_te_{\bu}^n\Vert _0\Vert\nabla d_t\hat{e}_{\bu}^{n+1}\Vert\\
			&\textstyle +\frac{C_0C_2(1+C_1)}\tau\left(\Vert e_{\bu}^{n-1}\Vert _0+\Vert e_{\bu}^n\Vert _0\right)\left(\int_{t_{n-1}}^{t_n}\Vert \bu_t\Vert _2dt\right)\Vert\nabla d_t\hat{e}_{\bu}^{n+1}\Vert  \\
			&\textstyle +C_0C_2(1+C_1)^3\left(\Vert\nabla e_{\bu}^{n-1}\Vert+\Vert\nabla e_{\bu}^n\Vert\right)\Vert\nabla d_t\hat{e}_{\bu}^n\Vert \Vert\nabla d_t\hat{e}_{\bu}^{n+1}\Vert.
		\end{aligned}  
	\end{equation}
	By noting that
	\begin{equation*}
		\begin{aligned}
			\textstyle d_t\boldsymbol{R}_{\bu}^{n+1} &\textstyle =\frac{1}{\tau^2}\left[\int_{t_n}^{t_{n+1}}(t-t_n)\bu_{tt}(t)dt-\int_{t_{n-1}}^{t_n}(t-t_{n-1})\bu_{tt}(t)dt\right] \\
			&\textstyle =\frac{1}{\tau^2}\int_{t_{n-1}}^{t_n}\int_r^{r+\tau}\int_s^{r+\tau}\bu_{ttt}(\xi)d\xi dsdr,
		\end{aligned}
	\end{equation*}
	we obtain
	\begin{equation}\label{dtRu}
		\textstyle(d_t\boldsymbol{R}_{\bu}^{n+1},d_t\hat{e}_{\bu}^{n+1})\leq\Vert d_t\boldsymbol{R}_{\bu}^{n+1}\Vert _{-1}\Vert d_t\hat{e}_{\bu}^{n+1}\Vert _1\leq(1+C_1)\left(\int_{t_{n-1}}^{t_{n+1}}\Vert \bu_{ttt}\Vert _{-1}dt\right)\Vert \nabla d_t\hat{e}_{\bu}^{n+1}\Vert .
	\end{equation}
	The second term on the right-hand side of \eqref{dteu} can be bounded by 
	\begin{equation}\label{dtep}
		\begin{aligned}
			&\textstyle-\frac12(\nabla d_t(2p(t_{n+1})-p^n-p^{n+1}),d_t\hat{e}_{\bu}^{n+1})\\
			\textstyle=&\textstyle\frac{\tau}{2}(\nabla d_t(e_p^{n+1}+p(t_{n+1})-p(t_n)+e_p^n),\nabla d_t(-e_p^{n+1}+p(t_{n+1})-p(t_n)+e_p^n))\\
			\textstyle=&\textstyle-\frac{\tau}{2}(\Vert\nabla d_te_p^{n+1}\Vert^2-\Vert\nabla d_te_p^n\Vert^2)+\tau(\nabla d_te_p^n,\nabla d_t(p(t_{n+1})-p(t_n)))+\frac{\tau}{2}\Vert\nabla d_t(p(t_{n+1})-p(t_n))\Vert^2\\
			\textstyle\leq&\textstyle-\frac{\tau}{2}(\Vert\nabla d_te_p^{n+1}\Vert^2-\Vert\nabla d_te_p^n\Vert^2)+\tau^2\Vert\nabla d_te_p^n\Vert^2+(\frac{\tau}{4}+\frac{\tau^2}{2})\int_{t_n}^{t_{n+1}}\Vert\nabla p_{tt}\Vert^2dt.
		\end{aligned}
	\end{equation}
	Combining  \eqref{S_1}-\eqref{dtep} with \eqref{dteu}, we obtain
	\begin{equation}\label{dteu:end}
		\begin{aligned}
			&\textstyle\frac{\Vert d_{t}e_{\bu}^{n+1}\Vert^2-\Vert d_{t}e_{\bu}^n\Vert^2}{2\tau}+\frac{\Vert d_{t}\hat{e}_{\bu}^{n+1}-d_{t}e_{\bu}^{n}\Vert^2}{2\tau}+\nu\Vert\nabla d_{t}\hat{e}_{\bu}^{n+1}\Vert^2+\frac{\tau}{2}(\Vert\nabla d_te_p^{n+1}\Vert^2-\Vert\nabla d_te_p^{n}\Vert^2)\\
			\textstyle\leq&\textstyle(1+C_1)\left(\int_{t_{n-1}}^{t_{n+1}}\Vert \bu_{ttt}\Vert _{-1}dt+2C_2M\int_{t_{n-1}}^{t_{n+1}}\Vert \bu_{tt}\Vert _0dt+2C_2\int_{t_{n-1}}^{t_{n+1}}\Vert \bu_t\Vert _1^2dt\right)\Vert \nabla d_t\hat{e}_{\bu}^{n+1}\Vert+\tau^2\Vert\nabla d_te_p^{n}\Vert^2\\
			&\textstyle+C_2(1+C_1)|d_te_Q^{n+1}|(\Vert \bu(t_n)\Vert _1^2+(1+C_1)^2\Vert\nabla\hat{e}_{\bu}^n\Vert^2)\Vert\nabla d_t\hat{e}_{\bu}^{n+1}\Vert+\frac{2C_2(1+C_1)M}\tau|e_Q^n|\left(\int_{t_{n-1}}^{t_n}\Vert \bu_t\Vert _0dt\right)\Vert\nabla d_t\hat{e}_{\bu}^{n+1}\Vert \\
			&\textstyle +\frac{C_0C_2(1+C_1)}\tau(\Vert e_{\bu}^{n-1}\Vert _0+\Vert e_{\bu}^n\Vert _0)\left(\int_{t_{n-1}}^{t_n}\Vert \bu_t\Vert _2dt\right)\Vert\nabla d_t\hat{e}_{\bu}^{n+1}\Vert+\frac{1+2\tau}{4}\tau\int_{t_n}^{t_{n+1}}\Vert\nabla p_{tt}\Vert^2dt  \\
			&\textstyle +2C_0C_2(1+C_1)M\Vert d_te_{\bu}^n\Vert _0\Vert\nabla d_t\hat{e}_{\bu}^{n+1}\Vert+C_0C_2(1+C_1)^3(\Vert\nabla e_{\bu}^{n-1}\Vert+\Vert\nabla e_{\bu}^n\Vert)\Vert\nabla d_t\hat{e}_{\bu}^n\Vert \Vert\nabla d_t\hat{e}_{\bu}^{n+1}\Vert\\
			\textstyle\leq&\textstyle\frac{\nu}{8}\Vert\nabla d_t\hat{e}_{\bu}^{n+1}\Vert^2+\frac{8}{\nu}(1+C_1)^2\left(\tau\int_{t_{n-1}}^{t_{n+1}}\Vert \bu_{ttt}\Vert _{-1}^2dt+4C_2^2M^2\tau\int_{t_{n-1}}^{t_{n+1}}\Vert \bu_{tt}\Vert^2dt+4C_2^2\tau\int_{t_{n-1}}^{t_{n+1}}\Vert \bu_t\Vert_1^4dt\right)+\tau^2\Vert\nabla d_te_p^{n}\Vert^2\\
			&\textstyle+\frac{16}{\nu}C_2^2(1+C_1)^2(M^4+(1+C_1)^4\Vert\nabla\hat{e}_{\bu}^n\Vert^4)|d_te_Q^{n+1}|^2+\frac{32C_2^2(1+C_1)^2M^2}{\nu\tau}\left(\int_{t_{n-1}}^{t_n}\Vert \bu_t\Vert^2dt\right)|e_Q^n|^2\\
			&\textstyle+\frac{16C_0^2C_2^2(1+C_1)^2}{\nu\tau}(\Vert e_{\bu}^{n-1}\Vert^2+\Vert e_{\bu}^n\Vert^2)\left(\int_{t_{n-1}}^{t_n}\Vert \bu_t\Vert _2^2dt\right)+\frac{32C_0^2C_2^2(1+C_1)^2M^2}{\nu}\Vert d_te_{\bu}^n\Vert^2\\
			&\textstyle+\frac{16C_0^2C_2^2(1+C_1)^6}{\nu}(\Vert\nabla e_{\bu}^{n-1}\Vert^2+\Vert\nabla e_{\bu}^n\Vert^2)\Vert\nabla d_t\hat{e}_{\bu}^n\Vert^2+\frac{1+2\tau}{4}\tau\int_{t_n}^{t_{n+1}}\Vert\nabla p_{tt}\Vert^2dt.
		\end{aligned}
	\end{equation}
	Noting that $d_te_Q^{n+1}$ is on the right-hand side of \eqref{dteu:end}, we should estimate it. Observing \eqref{Q:err:orig}, we have
	\begin{equation}\label{dteQ}
		\begin{aligned}
			\textstyle\theta|d_te_Q^{n+1}|\leq&\textstyle\frac{2\Vert\hat{e}_{\bu}^{n+1}-e_{\bu}^n\Vert^2}{\tau}+2\int_{t_n}^{t_{n+1}}\Vert\bu_t\Vert^2dt+2C_0b(\bu^n,\bu^n,\hat{\bu}^{n+1})\\
			\textstyle\leq&\textstyle\frac{2\Vert\hat{e}_{\bu}^{n+1}-e_{\bu}^n\Vert^2}{\tau}+2\int_{t_n}^{t_{n+1}}\Vert\bu_t\Vert^2dt+2C_0C_2M^2\left(\int_{t_n}^{t_{n+1}}\Vert\bu_t\Vert dt\right)\\
			&\textstyle+2C_0C_2M^2(1+C_1)(\Vert\nabla \hat{e}_{\bu}^{n+1}\Vert+2\Vert\nabla \hat{e}_{\bu}^n\Vert)+2C_0C_2(1+C_1)^2M(\Vert\nabla\hat{e}_{\bu}^{n}\Vert+2\Vert\nabla\hat{e}_{\bu}^{n+1}\Vert)\Vert\nabla\hat{e}_{\bu}^n\Vert\\
			&\textstyle+2C_0C_2(1+C_1)^3\Vert\nabla\hat{e}_{\bu}^n\Vert^2\Vert\nabla\hat{e}_{\bu}^{n+1}\Vert.
		\end{aligned}
	\end{equation}
	Similar to \eqref{p:Cauchy}, squaring both sides of \eqref{dteQ} and applying Cauchy-Schwarz inequality, we find
	\begin{equation}
		\begin{aligned}
			\textstyle\theta^2|d_te_Q^{n+1}|^2\leq&\textstyle C_5(\frac{\Vert\hat{e}_{\bu}^{n+1}-e_{\bu}^n\Vert^4}{\tau^2}+\tau\int_{t_n}^{t_{n+1}}\Vert\bu_t\Vert^4dt+\tau\int_{t_n}^{t_{n+1}}\Vert\bu_t\Vert^2 dt+\Vert\nabla \hat{e}_{\bu}^{n+1}\Vert^2\\
			&\textstyle\qquad+\Vert\nabla \hat{e}_{\bu}^n\Vert^2+\Vert\nabla\hat{e}_{\bu}^{n}\Vert^4+\Vert\nabla\hat{e}_{\bu}^{n+1}\Vert^2\Vert\nabla\hat{e}_{\bu}^n\Vert^2+\Vert\nabla\hat{e}_{\bu}^n\Vert^4\Vert\nabla\hat{e}_{\bu}^{n+1}\Vert^2),
		\end{aligned}
	\end{equation}
	where
	\begin{equation*}
		\textstyle C_5=8+4C_0^2C_2^2M^4+20C_0^2C_2^2(1+C_1)^2M^4+20C_0^2C_2^2(1+C_1)^4M^2+4C_0^2C_2^2(1+C_1)^6.
	\end{equation*}
	To estimate the right-hand side of \eqref{dteu:end} and note that $\tau<1$, we observe from Theorem \ref{velocity estimate} that
	\begin{equation}\label{pre:estimate}
		\left\{
		\begin{aligned}
			&\textstyle\tau\sum\limits_{i=0}^n\Vert\hat{e}_{\bu}^{i+1}-e_{\bu}^i\Vert^2\leq C\tau^2,\\
			&\textstyle\tau\sum\limits_{i=0}^n\Vert\nabla \hat{e}_{\bu}^{i+1}\Vert^2\leq C\tau^2,\\
			&\textstyle \max\{\Vert\nabla \hat{e}_{\bu}^n\Vert^2,\Vert\nabla \hat{e}_{\bu}^n\Vert^4\}\leq C,
		\end{aligned}
		\right.
	\end{equation}
	where $\textstyle C=\max\{\frac{C_6C_8}{\nu}+C_9,C_6C_8+C_9,(\frac{C_6C_8}{\nu}+C_9)^2\}$.
	
	Let $\theta^2\geq\frac{16}{\nu}C_2^2(1+C_1)^2C_5(M^4+(1+C_1)^4C)$, then we have 
	\begin{equation}\label{dteQ2}
		\begin{aligned}
			&\textstyle\frac{16}{\nu}C_2^2(1+C_1)^2(M^4+(1+C_1)^4\Vert\nabla\hat{e}_{\bu}^n\Vert^4)|d_te_Q^{n+1}|^2\\
			\textstyle\leq&\textstyle\frac{\Vert\hat{e}_{\bu}^{n+1}-e_{\bu}^n\Vert^4}{\tau^2}+\tau\int_{t_n}^{t_{n+1}}\Vert\bu_t\Vert^4dt+\tau\int_{t_n}^{t_{n+1}}\Vert\bu_t\Vert^2 dt+\Vert\nabla \hat{e}_{\bu}^{n+1}\Vert^2+\Vert\nabla \hat{e}_{\bu}^n\Vert^2+\Vert\nabla\hat{e}_{\bu}^{n}\Vert^4\\
			&\textstyle+\Vert\nabla\hat{e}_{\bu}^{n+1}\Vert^2\Vert\nabla\hat{e}_{\bu}^n\Vert^2+\Vert\nabla\hat{e}_{\bu}^n\Vert^4\Vert\nabla\hat{e}_{\bu}^{n+1}\Vert^2.
		\end{aligned}
	\end{equation}
	Summing \eqref{dteQ2} over $i=0$ to $n$ and collecting \eqref{pre:estimate} yields
	\begin{equation}\label{dteQ:tau^2}
		\textstyle\tau\sum\limits_{i=0}^n\frac{16}{\nu}C_2^2(1+C_1)^2(M^4+(1+C_1)^4\Vert\nabla\hat{e}_{\bu}^i\Vert^4)|d_te_Q^{i+1}|^2\leq 2(C^2\tau+M+2C)\tau^2.
	\end{equation}
	Taking the sum of \eqref{dteu:end} over $i=1,\cdots, n$ and multiplying $2\tau$, then combining \eqref{dteQ:tau^2}, we have
	\begin{equation}\label{dteu:sum}
		\begin{aligned}
			&\textstyle\Vert d_{t}e_{\bu}^{n+1}\Vert^2-\Vert d_{t}e_{\bu}^1\Vert^2+\sum\limits_{i=1}^n\Vert d_{t}\hat{e}_{\bu}^{i+1}-d_{t}e_{\bu}^{i}\Vert^2+\frac{7}{4}\nu\tau\sum\limits_{i=1}^n\Vert\nabla d_{t}\hat{e}_{\bu}^{i+1}\Vert^2+\tau^2(\Vert\nabla d_te_p^{n+1}\Vert^2-\Vert\nabla d_te_p^{1}\Vert^2)\\
			\textstyle\leq&~\textstyle\frac{16}{\nu}(1+C_1)^2\tau^2\left(\int_{0}^{T}\Vert \bu_{ttt}\Vert _{-1}^2dt+4C_2^2M^2\int_{0}^{T}\Vert \bu_{tt}\Vert^2dt+4C_2^2\int_{0}^{T}\Vert \bu_t\Vert_1^4dt\right)+2\tau^3\sum\limits_{i=1}^n\Vert\nabla d_te_p^{i}\Vert^2\\
			&\textstyle+4(C^2\tau+M+2C)\tau^3+\frac{64CC_2^2(1+C_1)^2M^2}{\nu}\tau^2\left(\int_{0}^{T}\Vert \bu_t\Vert^2dt\right)\\
			&\textstyle+\frac{64CC_0^2C_2^2(1+C_1)^2}{\nu}\tau^2\left(\int_{0}^{T}\Vert \bu_t\Vert _2^2dt\right)+\frac{64C_0^2C_2^2(1+C_1)^2M^2}{\nu}\tau\sum\limits_{i=1}^n\Vert d_te_{\bu}^i\Vert^2\\
			&\textstyle+\frac{64CC_0^2C_2^2(1+C_1)^6}{\nu}\tau^2\sum\limits_{i=1}^n\Vert\nabla d_t\hat{e}_{\bu}^i\Vert^2+\frac{1+2\tau}{2}\tau^2\int_{0}^{T}\Vert\nabla p_{tt}\Vert^2dt.
		\end{aligned}
	\end{equation}
	Since $\hat{e}_{\bu}^0=e_{\bu}^0=0$, we have $\tau^2\Vert\nabla d_t\hat{e}_{\bu}^1\Vert^2=\Vert\nabla\hat{e}_{\bu}^1\Vert^2$. Using Gr\"onwall inequality with $a_n=\Vert d_{t}e_{\bu}^{n+1}\Vert^2+\tau^2\Vert\nabla d_te_p^{n+1}\Vert^2$ and let $\tau\leq \frac{3\nu^2}{256C_0^2C_2^2(1+C_1)^6}$, we find
	\begin{equation}\label{dteu:tau:end}
		\begin{aligned}
			&\textstyle\Vert d_{t}e_{\bu}^{n+1}\Vert^2+\sum\limits_{i=1}^n\Vert d_{t}\hat{e}_{\bu}^{i+1}-d_{t}e_{\bu}^{i}\Vert^2+\nu\tau\sum\limits_{i=1}^n\Vert\nabla d_{t}\hat{e}_{\bu}^{i+1}\Vert^2+\tau^2\Vert\nabla d_te_p^{n+1}\Vert^2\\
			\textstyle\leq&\textstyle\Vert d_{t}e_{\bu}^1\Vert^2+\tau^2\Vert\nabla d_te_p^{1}\Vert^2+\frac{3}{4}\nu\tau^2\Vert\nabla d_t\hat{e}_{\bu}^1\Vert^2+C_{10}C_{11}\tau^2\\
			\textstyle\leq&\textstyle\Vert d_{t}e_{\bu}^1\Vert^2+\tau^2\Vert\nabla d_te_p^{1}\Vert^2+\nu\Vert\nabla \hat{e}_{\bu}^1\Vert^2+C_{10}C_{11}\tau^2,
		\end{aligned}
	\end{equation}
	where
	\begin{equation*}
		\begin{aligned}
			\textstyle C_{10}=&\textstyle \exp\left(\frac{64C_0^2C_2^2(1+C_1)^2M^2}{\nu}T+2T\right),\\
			\textstyle C_{11}=&\textstyle\frac{16}{\nu}(1+C_1)^2\left(\int_{0}^{T}\Vert \bu_{ttt}\Vert _{-1}^2dt+4C_2^2M^2\int_{0}^{T}\Vert \bu_{tt}\Vert^2dt+4C_2^2\int_{0}^{T}\Vert \bu_t\Vert_1^4dt\right)+4(C^2\tau+M+2C)\\
			&\textstyle+\frac{64CC_2^2(1+C_1)^2M^2}{\nu}\left(\int_{0}^{T}\Vert \bu_t\Vert^2dt\right)+\frac{64CC_0^2C_2^2(1+C_1)^2}{\nu}\left(\int_{0}^{T}\Vert \bu_t\Vert _2^2dt\right)+\frac{1+2\tau}{2}\int_{0}^{T}\Vert\nabla p_{tt}\Vert^2dt.
		\end{aligned}
	\end{equation*}
	To bound $\Vert d_{t}e_{\bu}^1\Vert^2+\tau^2\Vert\nabla d_te_p^{1}\Vert^2+\frac{3}{4}\nu\tau^2\Vert\nabla d_t\hat{e}_{\bu}^1\Vert^2$ in \eqref{dteu:tau:end}, we let $n = 0$ in \eqref{err:mom:L2} and obtain
	\begin{equation}
		\begin{aligned}
			&\textstyle\Vert\hat{e}_{\bu}^{1}\Vert^2+\nu\tau\Vert\nabla\hat{e}_{\bu}^{1}\Vert^2\\
			\textstyle=&\textstyle-\tau b(\bu(t_{1}),\bu(t_{1}),\hat{e}_{\bu}^{1})+\mathcal{Q}^{1}\tau b(\bm u^0,\bm u^0,\hat{e}_{\bu}^{1})-\tau (\nabla(p(t_{1})-p(t_0)),\hat{e}_{\bu}^{1})-\tau (\boldsymbol{R}_{\bu}^{1},\hat{e}_{\bu}^{1}).
		\end{aligned}
	\end{equation}
	Using similar arguments as in \eqref{err:velo:add:1}-\eqref{err:velo:sum:1} and \eqref{u:Q:err:1}, and combining the regularity of $\bu$ , we deduce
	\begin{equation}\label{err:eu1}
		\begin{aligned}
			&\textstyle\Vert\hat{e}_{\bu}^{1}\Vert^2+\nu\tau\Vert\nabla\hat{e}_{\bu}^{1}\Vert^2\\
			\textstyle\leq&\textstyle\frac{1}{2}\Vert\hat{e}_{\bu}^{1}\Vert^2+2\tau^4\left(\Vert\bu_{tt}\Vert_{L^\infty(0,t_1)}^2+4C_2^2\Vert\nabla\bu_t\Vert_{L^2(0,t_1)}^2+\Vert\nabla p_t\Vert_{L^\infty(0,t_1)}^2\right)+2C_2^2M^4\tau^2|e_Q^1|^2\\
			\textstyle\leq&\textstyle\frac{1}{2}\Vert\hat{e}_{\bu}^{1}\Vert^2+C_{12}\tau^4,
		\end{aligned}
	\end{equation}
	where we define 
	\begin{equation*}
		\textstyle C_{12}=2\left(\Vert\bu_{tt}\Vert_{L^2(0,t_1)}^2+4C_2^2\Vert\nabla\bu_t\Vert_{L^2(0,t_1)}^2+\Vert\nabla p_t\Vert_{L^2(0,t_1)}^2+2C_2^2M^4C^\star\right).
	\end{equation*}
	We can derive from \eqref{err:proj:grad} with $n=1$ that
	\begin{equation}\label{err:dtep1}
		\textstyle \tau^2\Vert\nabla d_te_p^1\Vert^2\leq\tau^{-2}(\Vert e_{\bu}^1\Vert^2+\Vert\hat{e}_{\bu}^1\Vert^2)+\tau^2\Vert\nabla d_tp(t_1)\Vert^2\leq (4C_{12}+\Vert\nabla d_tp(t_1)\Vert^2)\tau^2.
	\end{equation}
	Combining the above estimates with \eqref{dteu:tau:end}, we finally obtain
	\begin{equation}\label{dteu:final}
		\begin{aligned}
			&\textstyle\Vert d_{t}e_{\bu}^{n+1}\Vert^2+\sum\limits_{i=1}^n\Vert d_{t}\hat{e}_{\bu}^{i+1}-d_{t}e_{\bu}^{i}\Vert^2+\nu\tau\sum\limits_{i=1}^n\Vert\nabla d_{t}\hat{e}_{\bu}^{i+1}\Vert^2+\tau^2\Vert\nabla d_te_p^{n+1}\Vert^2\\
			\textstyle\leq&\textstyle\left((6+\tau)C_{12}+\Vert\nabla p_t\Vert_{L^2(0,t_1;L^2{\Omega})}^2+C_{10}C_{11}\right)\tau^2,
		\end{aligned}
	\end{equation}
	which implies in particular
	\begin{equation}\label{eu-eu}
		\textstyle\Vert e_{\bu}^{n+1}-e_{\bu}^n\Vert=\tau\Vert d_te_{\bu}^{n+1}\Vert\leq \sqrt{(6+\tau)C_{12}+\Vert\nabla p_t\Vert_{L^2(0,t_1;L^2{\Omega})}^2+C_{10}C_{11}}\tau^2.
	\end{equation}
	Hence thanks to Theorem \ref{velocity estimate} and \eqref{p:sum}, \eqref{eu-eu}, we can derive from the above that
	\begin{equation*}\label{p:final}
		\begin{aligned}
			&\tau\sum\limits_{i=0}^n\Vert e_p^{i+1}\Vert^2\\
			&\leq\textstyle\hat{C}\tau(\frac{1}{\tau^2} \sum\limits_{i=0}^n\Vert e_{\bu}^{i+1} - e_{\bu}^i\Vert_{-1}^2+\nu\sum\limits_{i=0}^n\Vert\nabla \hat{e}_{\bu}^{i+1}\Vert^2+\tau\int_{0}^{t_{n+1}} \Vert\bu_{tt}\Vert_{-1}^2 dt+\tau\int_{0}^{t_{n+1}}\Vert\bu_t\Vert^2 dt\\
			&\textstyle\qquad\qquad\qquad+\sum\limits_{i=0}^n\Vert e_{\bu}^i\Vert^2+\sum\limits_{i=0}^n\Vert\nabla e_{\bu}^i\Vert^4+\sum\limits_{i=0}^n|e_Q^{i+1}|^2)\\
			\textstyle&\leq\textstyle\hat{C}((6+\tau)C_{12}+\Vert\nabla p_t\Vert_{L^2(0,t_1;L^2(\Omega))}^2+C_{10}C_{11})T+(C_6C_8+C_9)(1+2T)+2M+(C_6C_8+C_9)^2T)\tau^2,
		\end{aligned}
	\end{equation*}
	provided that $\tau,~\theta$ satisfies the following condition $\tau\leq \max\{\frac{3\nu^2}{256C_0^2C_2^2(1+C_1)^6},\tau_0\}$ and $\theta\geq\max\{8(1+C_0),4C_2(1+C_1)\sqrt{\frac{C_5}{\nu}(M^4+(1+C_1)^4C)}\}$.
	
	The proof is complete.
\end{proof}
\section{Numerical experiments}\label{num}
In this section, we will demonstrate the accuracy of the P-DRLM1 scheme with pressure correction for the NS equations through numerical experiments. All the tests are implemented on the  finite element software package FreeFEM++\cite{hecht2012new}.
\subsection{Convergence test}
We first carry out the convergence test to show the first-order temporal accuracy by the lattice-vortex problem \cite{MajdaBertozzi,belding2022efficient,SchroederLube}, whose true solution is given as follow:
\begin{equation*}
	\left \{
	\begin{aligned}
		&\textstyle u(x,y,t)= sin(2 \pi x) sin(2 \pi y) exp(-8 \nu\pi^2 t),\\
		&\textstyle v(x,y,t)=cos(2 \pi x)cos(2\pi y) exp(-8\nu \pi^2 t),\\
		&\textstyle p(x,y,t)=\frac{1}{2}(1- sin^2(2 \pi x)- cos^2(2 \pi  y)) exp(-16\nu\pi^2 t).
	\end{aligned}
	\right.
\end{equation*}
so that the source term $\bff=0$. The spatial domain is set as $\Omega = (0,1)^2$ with the Dirichlet boundary condition and the terminal time $T=1$. We fix the stabilizing parameter of P-DRLM1 as the standard choice $\theta=1$ and the viscosity is chosen as $\nu=10^{-1}$. 
The spatial discretization is based on the finite element method. Due to the inf-sup condition is applied to show the optimal error estimate for the pressure, we use Taylor-Hood element (P2-P1) \cite{BrennerScott} for the velocity and pressure pair and fix the spatial step size as $h=0.01$. The temporal step size is decreased by a factor \textcolor{blue}{ of 2,} and its corresponding results are present in Table \ref{tab-conv:m0}. The first-order convergent rate can be easily observed both for the velocity and pressure, which are consistent with the error estimates in Theorems \ref{velocity estimate} and \ref{pressure estimate}.
\begin{table}[!ht]
	\centering
	\begin{tabular}{|c|cccccc|}
		\hline
		$\tau$ & $\|e_{\bu}\|$ & Rate & $\|e_Q\|$ & Rate & $\Vert e_{p}\Vert$ & Rate \\
		\hline
		$1/32$ & 1.1836e-04 & - & 1.6093e-02 & - & 3.6038e-02 & - \\
		\hline
		$1/64$ & 5.1654e-05 & 1.1962 & 7.9698e-03 &1.0138  & 1.7125e-02 & 1.0735 \\
		\hline
		$1/128$ & 2.3950e-05 & 1.1089 & 3.9278e-03 & 1.0208 & 8.1238e-03 & 1.0759 \\
		\hline
		$1/256$ & 1.1508e-05 & 1.0574 & 1.9459e-03 & 1.0133 & 3.9219e-03 & 1.0506 \\
		\hline
	\end{tabular}
	\caption{Numerical results on the errors of the simulated velocity and pressure at the terminal time $T = 1$ produced by the proposed P-DRLM1 scheme  for the lattice-vortex problem with the viscosity $\nu=10^{-1}$. The stabilizing parameter $\theta=1$ and we also set the spatial step size to be fixed as $h=0.01$.}
	\label{tab-conv:m0}
\end{table}
\subsection{Lid-driven cavity flow}
Next, we evaluate the robustness of the proposed P-DRLM1 scheme via a realistic and challenging physical simulation, the well-known lid-driven cavity flow problem \cite{1982High,JuWang2017,LiRui2018SIAM,Li2022MOC}. For the 2D case, the computational domain is defined as $\Omega=(0,1)^2$, bounded by three stationary walls (at $x = 0$, $x = 1$, and $y = 0$) with no-slip boundary conditions, and a moving lid (at $y = 1$) with a tangential unit velocity. We set $Re=5000$ and $\theta=100$, and employ the P-DRLM1 scheme with a spatial resolution of $h=0.01$ {\color{magenta} and time step $\tau=0.002$}. Contour plots of the velocity magnitude at $t=2, 4, 8, 10, 20$, and $80$ are presented in Fig.\ref{fig-mag-cont}. The method is observed to be stable and robust.
%It can be seen that our method is quite stable and robust. 
Moreover, 
the plots agree well with the results reported in \cite{Doan2025JCP}, which were obtained  by the second-order numerical scheme.
In addition, we plot the velocity on the center line at $t=80$ compared with the benchmark results in Figure \ref{fig-mag-cont:mid}. As we can see, our velocity is very close to the benchmark results \cite{1982High}, especially in the turning points, like the $x-$component around $y=0.1$ and $y=0.97$ and $y-$component around $x=0.1$ and $x=0.96$. Hence, it demonstrates that the proposed scheme P-DRLM1 accurately captures the dynamical evolution of the velocity field,. 
\begin{figure}[htbp]
	\centering
	\begin{minipage}{0.32\textwidth}
		\centering
		\includegraphics[width=\textwidth]{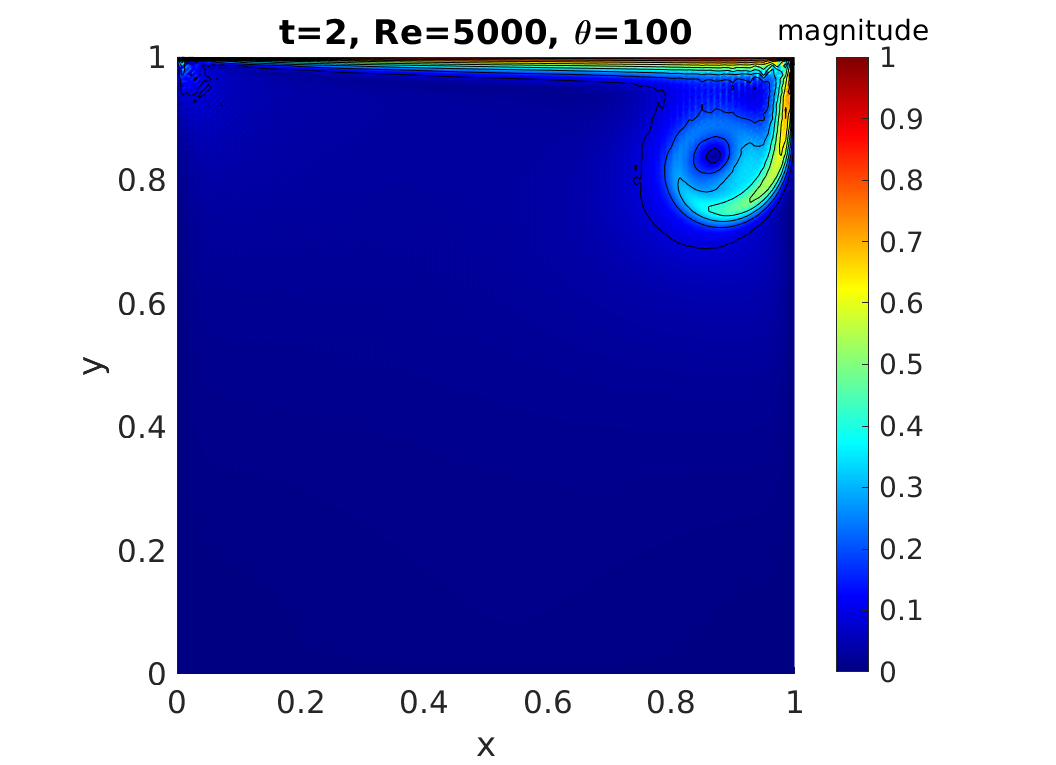}
	\end{minipage}
	\hfill
	\begin{minipage}{0.32\textwidth}
		\centering
		\includegraphics[width=\textwidth]{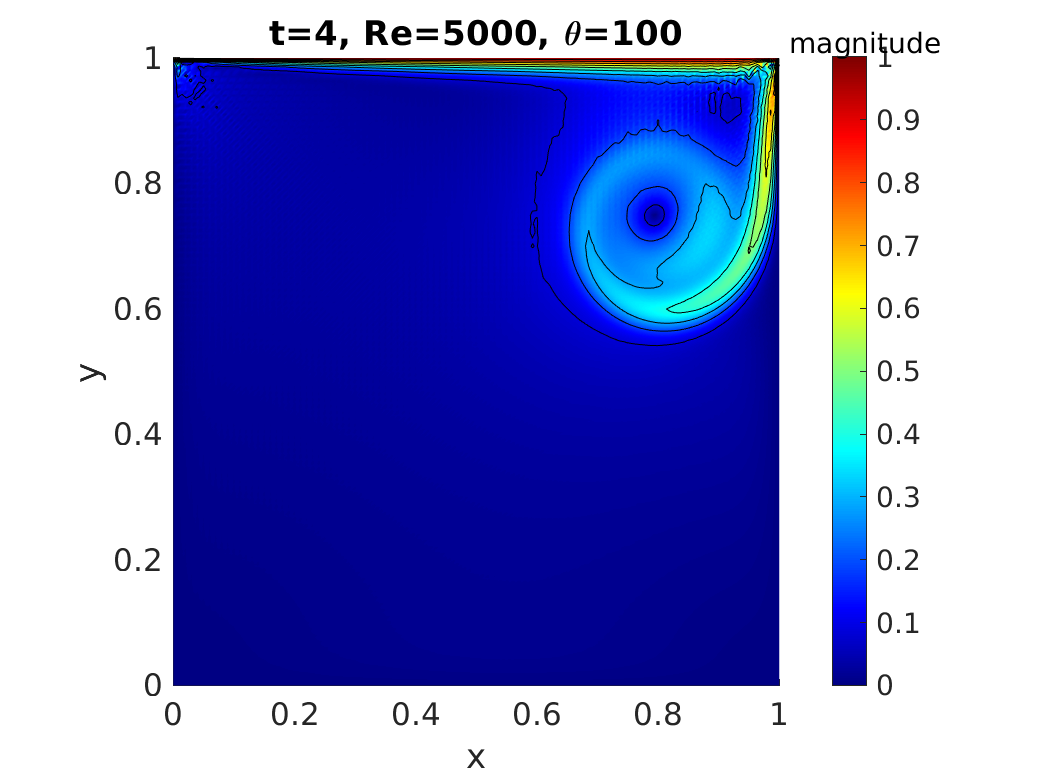}
	\end{minipage}
	\hfill
	\begin{minipage}{0.32\textwidth}
		\centering
		\includegraphics[width=\textwidth]{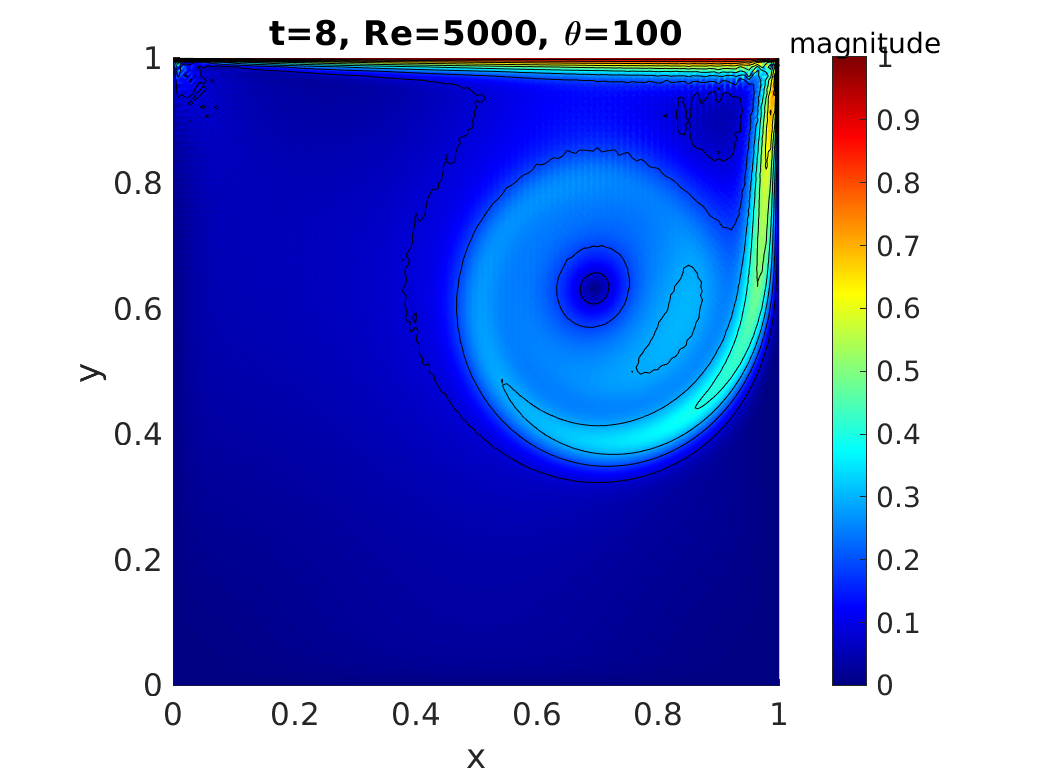}
	\end{minipage}
	
	\vspace{0.3cm}
	
	\begin{minipage}{0.32\textwidth}
		\centering
		\includegraphics[width=\textwidth]{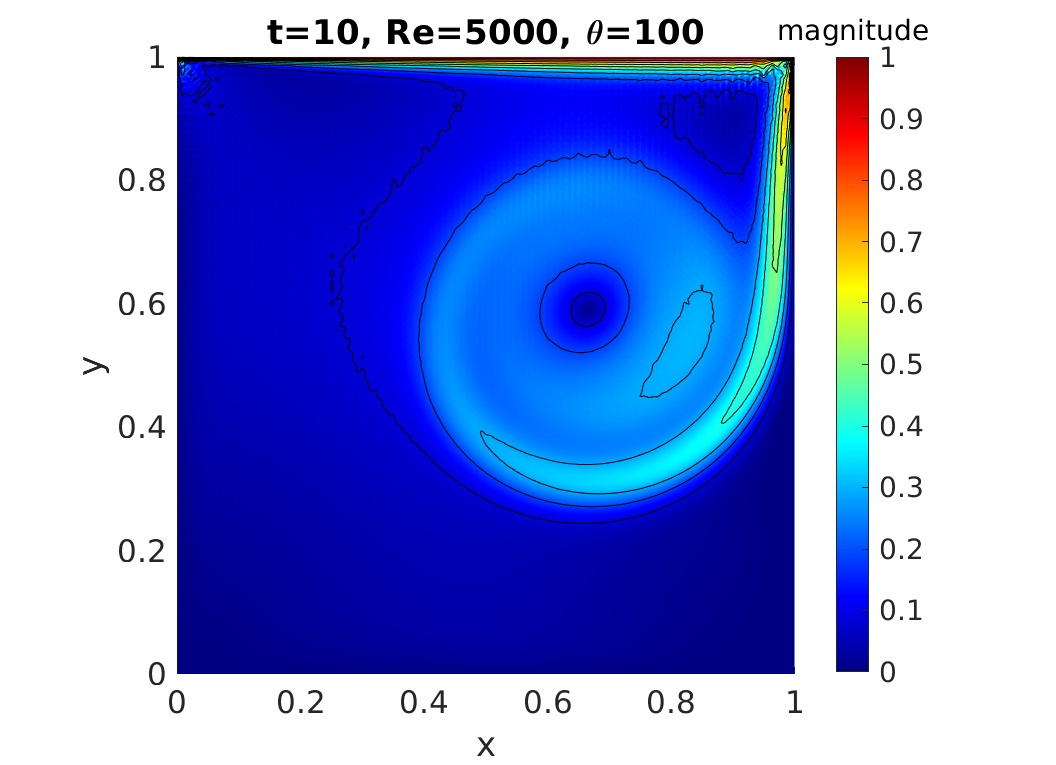}
	\end{minipage}
	\hfill
	\begin{minipage}{0.32\textwidth}
		\centering
		\includegraphics[width=\textwidth]{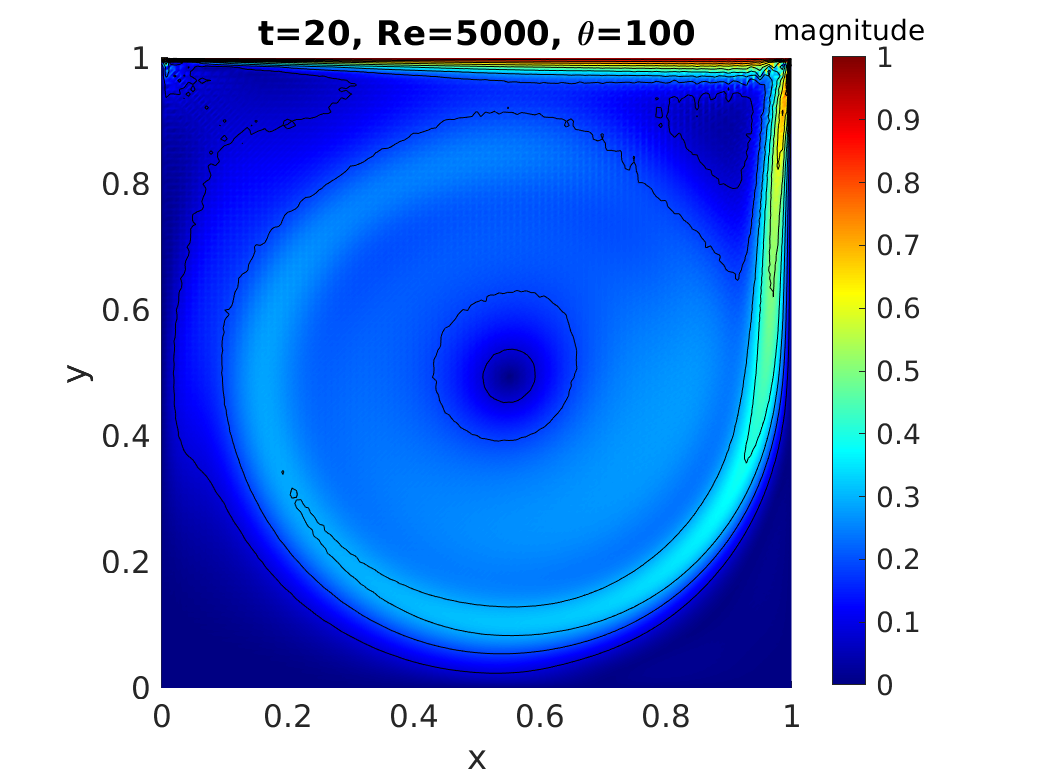}
	\end{minipage}
	\hfill
	\begin{minipage}{0.32\textwidth}
		\centering
		\includegraphics[width=\textwidth]{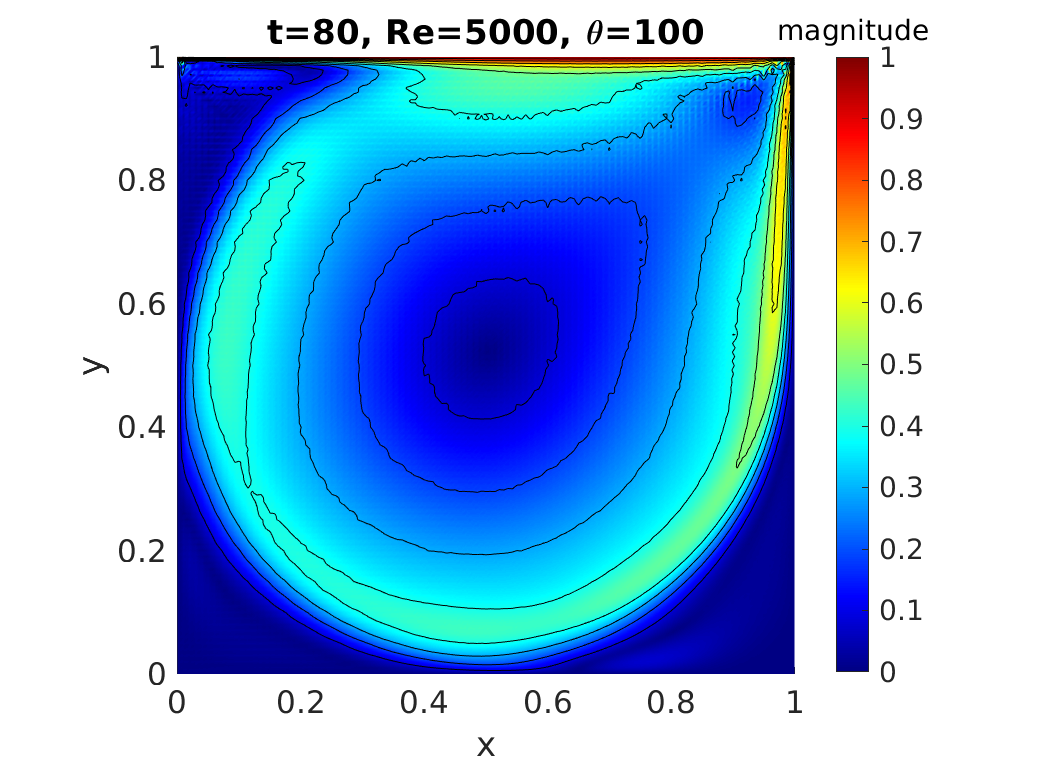}
	\end{minipage}
	
	\caption{Contour plots of the velocity magnitude at times $t=2, 4, 8, 10, 20$, and $80$ by P-DRLM1 scheme.}
	\label{fig-mag-cont}
\end{figure}

\begin{figure}[htbp]
	\centering
	\includegraphics[width=0.45\textwidth]{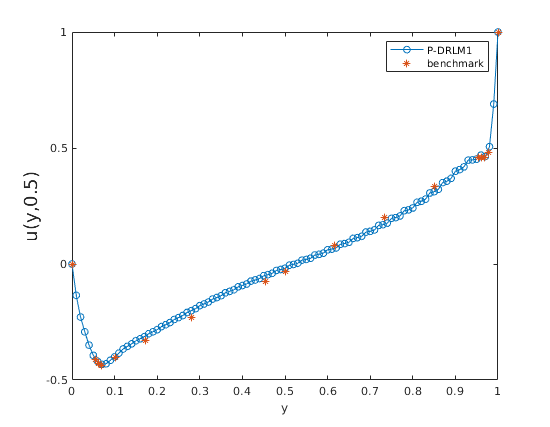}
	\includegraphics[width=0.48\textwidth]{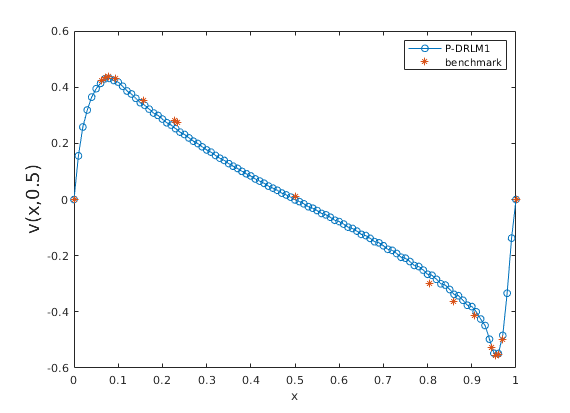} 
	\caption{The velocity at the center line with $x$-component velocity at $x = 0.5$ (left) and $y$-component velocity at $y = 0.5$ (right) for $\text{Re} = 5000$ at $t=80$.}
	\label{fig-mag-cont:mid}
\end{figure}

%The numerical results demonstrate that the proposed scheme accurately captures the dynamical evolution of the velocity field, and the final steady state is in excellent agreement with the benchmark data \cite{1982High}.

\section{Conclusion}\label{sect:con}
In this work, we established a rigorous temporal error analysis for the first-order pressure-correction DRLM scheme applied to the Navier-Stokes equations. By introducing a Lagrange multiplier to linearize the nonlinear terms and employing a projection method to decouple pressure and velocity, we derived optimal convergence rates for both variables. Numerical experiments confirmed the theoretical estimates.

While higher-order schemes are often more desirable in engineering applications, their analysis remains significantly more challenging and will be explored in future work. Specifically, we plan to extend our error analysis to the second-order P-DRLM2 scheme. Additionally, we aim to adapt the P-DRLM1 and P-DRLM2 methods to more complex systems, such as magnetohydrodynamics (MHD) and two-phase flows.

\section*{Acknowledgement}
%%%%%%%%%%%%%%%%%%%%%%%%%%
R. Lan's work is partially supported by National Natural Science Foundation of China under grant number 12301531, Shandong Provincial Natural Science Fund for Excellent Young Scientists Fund Program (Overseas) under grant number 2023HWYQ-064, Shandong Provincial Youth Innovation Project under the grant number 2024KJN057, and  the OUC Scientific Research Program for Young Talented Professionals.
H. Wang's work is partially supported by National Natural Science Foundation of China under grant number 12101526 and Young Elite Scientists Sponsorship Program by CAST 2023QNRC001.

\bibliographystyle{abbrv} 
\bibliography{reference}
\end{document}